\documentclass[11pt]{article}
\usepackage{amsmath,amsfonts,amssymb}
\usepackage{graphicx, color}

\newtheorem{theorem}{Theorem}[section]
\newtheorem{lemma}[theorem]{Lemma}
\newtheorem{proposition}[theorem]{Proposition}

\newtheorem{remark}[theorem]{Remark}
\newtheorem{example}[theorem]{Example}

\def\sqr#1#2{\vbox{\hrule height .#2pt
\hbox{\vrule width .#2pt height #1pt \kern #1pt
\vrule width .#2pt}\hrule height .#2pt }}
\def\square{\sqr74}
\def\endproof{\hphantom{MM}\hfill\llap{$\square$}\goodbreak}

\def\iff{\Longleftrightarrow}

\def\implies{\Longrightarrow}
\def\begi{\begin{itemize}}
\def\endi{\end{itemize}}

\def\Hat{\widehat}

\def\bega{\begin{array}}
\def\enda{\end{array}}

\def\dint{\int\!\!\int}
\def\Q{{\cal Q}}

\def\forall{\hbox{for every~ }}
\def\ov{\overline}
\def\ds{\displaystyle}
\def\R{\mathbb{R}}

\def\L{{\bf L}}

\def\forall{\hbox{for all}~}

\def\vp{\varphi}

\def\F{{\mathcal F}}
\def\O{{\mathcal O}}
\def\C{{\mathcal C}}

\def\Tilde{\widetilde}

\def\bfy{{\bf y}}

\def\argmax{\hbox{arg}\!\max}

\def\v{\vskip 1em}

\def\vsk{\vskip 3em}
\def\ve{\varepsilon}

\def\bel{\begin{equation}\label}
\def\eeq{\end{equation}}

\newcommand\N{{\mathbb N}\,}

\begin{document}
\title{\bf On the  Nash 
Equilibrium for a Model of Traffic Flow}

\author{Alberto Bressan$^{(*)}$, Daniel Lee Harmon$^{(*)}$ and  Elio Marconi$^{(**)}$\\~~\\
 {\small $^{(*)}$~Department of Mathematics, Penn State University.}
 \\
 {\small $^{(**)}$~Department of Mathematics 
 ``Tullio Levi-Civita", University of Padova.}\\~~\\
{\small E-mails: axb62@psu.edu,~~dharmon@psu.edu,~~elio.marconi@unipd.it}
}
\maketitle
\v
\begin{abstract} The paper consider a conservation law model of traffic flow on a highway, 
where each driver chooses an optimal departure time $\tau^d$, based on a departure cost 
$\vp(\tau^d)$ and an arrival cost $\psi(\tau^a)$. Nash equilibria are studied, where
no driver can reduce his own cost by changing his departure time.

The first results provide continuity estimates on the arrival time, depending on the distribution of departures.
In terms of the $\L^1$ distance, we show that the corresponding  map is H\"older continuous with exponent $1/2$
in Lagrangian variables (labelling individual drivers), and Lipschitz continuous in Eulerian space variables.

Concerning the regularity of the a Nash equilibrium, our results show that, when the total number of drivers is sufficiently large, a queue alway forms at the entrance of the highway and the corresponding traffic density contains a shock.
Moreover, for any number of drivers (possibly small), 
we prove that for generic $\C^2$ cost functions $\vp, \psi$ the Nash equilibrium solution still contains a shock.

In the last section of the paper, assuming that the initial departure rate is not a Nash equilibrium, 
we consider an evolution equation for the departure rate, where each driver keeps changing his
daily departure time  in order to reduce his own cost.   It is proved that this nonlocal integro-differential evolution equation admits a unique solution, globally in time. This provides some 
theoretical support to earlier numerical  simulations, which indicated
the dynamic instability of the Nash equilibrium and the possible existence of a chaotic attractor.
\end{abstract}

\maketitle


\section{Introduction}
\label{sec:1}
\setcounter{equation}{0}
Following a model introduced in \cite{BH1} we consider a group of drivers that, on a daily basis,
travel along a highway: starting from a residential neighborhood and reaching a workplace
in the city center.   As
in the classic Lighthill-Whitham-Richards model of traffic flow \cite{LW, R}, the traffic density $\rho(\tau,\xi)$ 
on the highway will be described by
the conservation law
\bel{1}
\rho_\tau+\bigl(\rho v(\rho)\bigr)_\xi~=~0.\eeq
Here  $\tau$ is time, $\xi$ is the space variable labeling points on the highway, and 
$v(\rho)$ is a decreasing function describing how the velocity of cars depends on the traffic density.

Given a cost $\vp(\tau^{d})$ for early departure
and a cost $\psi(\tau^{a})$ for late arrival,  we consider the optimization problem
for each driver, who needs to choose his  departure time in order to minimize the sum of the two costs.
Of course, the arrival time depends on the overall traffic density, and hence on the departure 
choices of all other drivers.  We say that a departure rate is a Nash equilibrium if no driver
can lower his total cost by choosing a different departure time. In this setting, 
the existence and uniqueness of the Nash equilibrium was proved in \cite{BH1}. 

In the present paper, several new results in this direction will be given.  
In Section~\ref{sec:2} we introduce the
basic definitions and assumptions, and review the construction of solutions to the conservation law via the Lax formula.

Sections~\ref{sec:3} and \ref{sec:4} analyze the departure-to-arrival time map: $\tau^d\mapsto\tau^a$
in Lagrangian and Eulerian variables, respectively.   More precisely, let $\kappa$ be the total number of drivers
and let $\beta\in[0,\kappa]$ be the Lagrangian variable labeling one particular driver.
In Theorem~\ref{t:31} we consider the map $\tau^d(\beta)\mapsto\tau^a(\beta)$.  By the Lax formula, 
this corresponds to a sup-convolution.  A detailed analysis shows
that the departure-to-arrival time  map is H\"older continuous with exponent $1/2$ as a map from $\L^1\bigl([0,\kappa]\bigr)$
into itself.    On the other hand, in Theorem~\ref{t:41} we consider a departure rate $\bar u\in \L^1(\R)$, so that
$\int_a^b \bar u(x) dx$ yields the amount of drivers who depart at a time $x\in [a,b]$.   Calling $y^{\bar u}(x)$
the arrival time of a driver departing at time $x\in\R$, we prove that the map $\bar u(\cdot)\mapsto y^{\bar u}(\cdot)$,
$$\hbox{[departure rate]} \quad\mapsto\quad \hbox{[arrival time function]}$$
is Lipschitz continuous w.r.t.~the $\L^1$ distance.   The proof is first achieved for a class of piecewise constant departure rates, where a homotopy argument can be used. Then it is extended by continuity to all $\bar u$
with bounded $\L^1$ norm and bounded support.

Section~\ref{sec:5} reviews the concept of Nash equilibrium departure rate, and provides an explicit example
where this departure rate is smooth. 
In Section~\ref{sec:6} we derive some properties of the Nash equilibrium,  and show some numerical simulations.

The following Sections~\ref{sec:7} and \ref{sec:8}  analyze the possible emergence of  shocks
 in the Nash equilibrium solution.  According to Proposition~\ref{p:71},
if the total number of drivers is large enough, then
a queue necessarily forms at the entrance of the highway.  
At the time when the queue is depleted, a shock appears in
the solution.
On the other hand, if the total number of drivers is sufficiently small, no queue is formed.  
However, this does not rule out the later emergence of shocks.  

Section~\ref{sec:8} begins by establishing the  one-sided Lipschitz
continuity of the  Nash equilibrium departure rate. It turns out that this regularity is optimal.
Indeed, as shown by Theorem~\ref{t:82},
for {\em generic} departure and arrival cost functions $\vp, \psi\in \C^2$ (i.e., ranging in an open dense set), 
the traffic pattern determined by the Nash equilibrium contains a shock, no matter how small is
the total number of drivers.

Finally, in Section~\ref{sec:9} we consider an evolutionary model for the departure rate $\bar u$, originally introduced in \cite{BLSY}.
Roughly speaking, assume that on day 1 the departure rate
differs from the equilibrium.  Then the following day some of the drivers may decide to change their departure time,
seeking a lower individual cost.  If this pattern is repeated day after day, will the departure rate oscillate chaotically
or will it approach the unique Nash equilibrium?   
As a first step to understand this challenging problem, 
in Theorem~\ref{t:91} we prove the well posedness of the integro-differential evolution equation (\ref{mod2}).  Namely,
 for any initial departure rate $\bar u\in \L^1\cap \L^\infty$ with bounded support, a unique solution exists, depending continuously on the initial data in the $\L^1$ distance.   This provides theoretical support to the 
simulations performed in \cite{BLSY}, which suggest the dynamic instability of the Nash equilibrium
and the possible existence of a chaotic attractor.

For an introduction to conservation laws we refer to \cite{Bbook, GMC, HR}. More generally, for
models of traffic flow also on networks we refer to \cite{GP}.  See \cite{FH1, FH2}
for an alternative perspective with engineering applications.   Results on the structure of generic solutions
to conservation laws can be found in \cite{BN, S}.

\section{The basic model}\label{sec:2}
\setcounter{equation}{0}
Throughout the following we focus on the boundary value problem, where the incoming flux of cars
at the beginning of the highway is assigned as a function of time (see Fig.~\ref{f:tf270}).    This can be 
formulated as a Cauchy problem, by switching the roles of the variables.  Together with (\ref{1})
we thus consider the conservation law
\bel{claw1}
u_t +f(u)_x~=~0
\qquad\qquad (t,x) \in [0,L]\times \R,\eeq
where $x$ corresponds to the previous time variable, 
while
$t\in [0,L]$ describes points along a highway of total length $L$.
The dependent variable $u=u(t,x)$
 denotes the flux of cars at time $x$ across the point $t$
 on the highway (see Fig.~\ref{f:tf270}).
 
\begin{figure}[htb]
   \centering
 \includegraphics[width=0.7\textwidth]{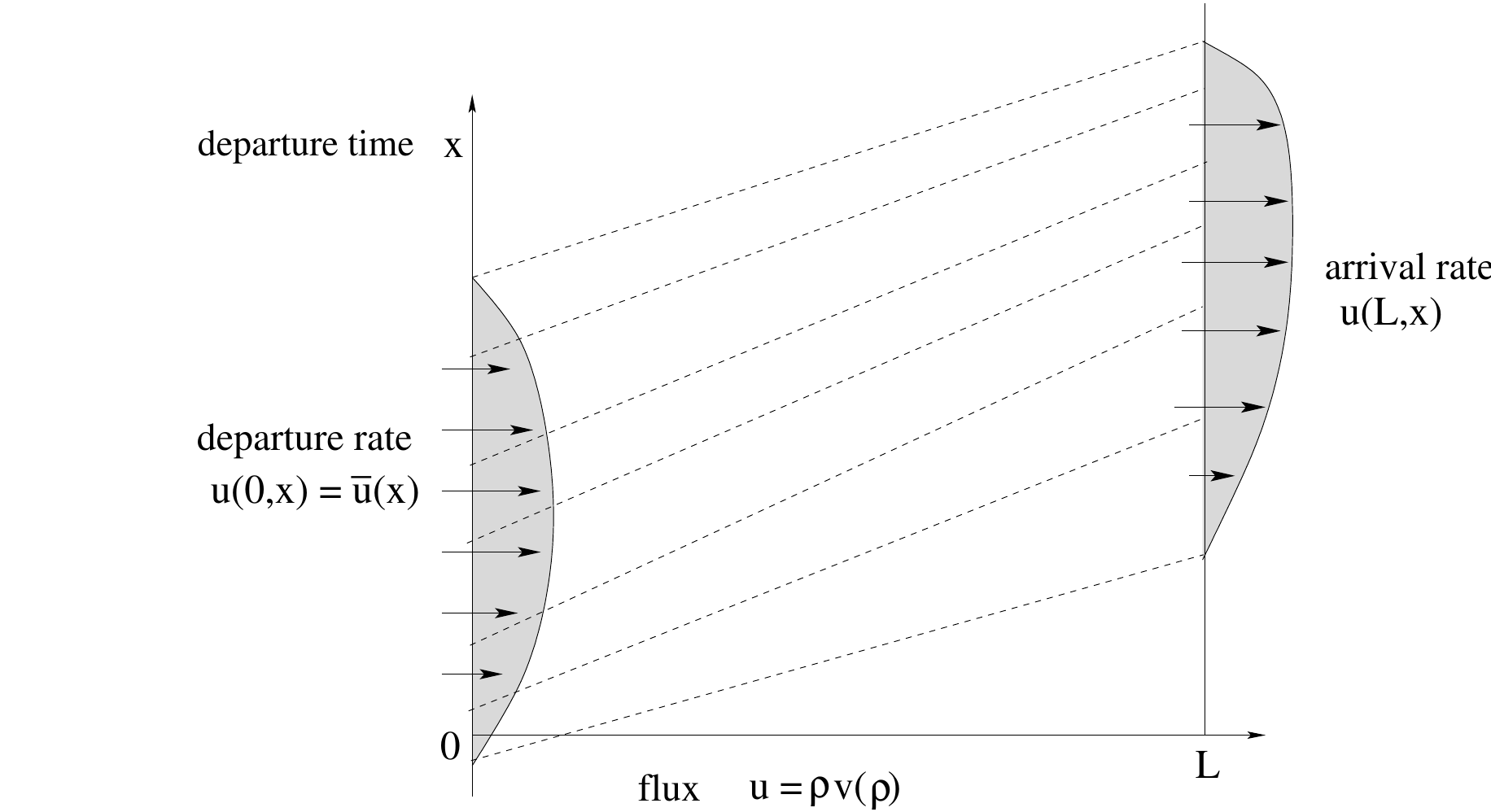}
   \caption{\small A departure rate and the corresponding arrival rate of cars
   traveling along a highway of length $L$. }
   \label{f:tf270}
\end{figure}

 The function $u\mapsto f(u) = \rho$
 is a partial inverse of the smooth, concave function $\rho\mapsto \rho v(\rho) =u$
 (see Fig.~\ref{f:t44}). 
 This is defined for $u\in [0,M]$, where
$$M~=~\max_\rho \,\rho\, v(\rho)$$
is the maximum possible flux for the conservation law (\ref{1}).
Throughout the following we assume
\begin{itemize}
\item[{\bf (A1)}] {\it The flux function $f:[0,M]\mapsto\R$
is continuous, increasing, and strictly convex. It is three times continuously differentiable  on the half-open interval $[0,\,M[\,$
and satisfies
\bel{fprop}
f(0)=0\,,\quad f'(0)\,=\, {1\over v(0)}\,>\,0,\quad f'(M-)=+\infty,\quad
 f''(u)\geq \omega_0 >0\quad \hbox{for}~~0\leq u<M,
\eeq
for some constant $\omega_0$.
}
\end{itemize}
 \begin{figure}[htb]
   \centering
 \includegraphics[width=1.0\textwidth]{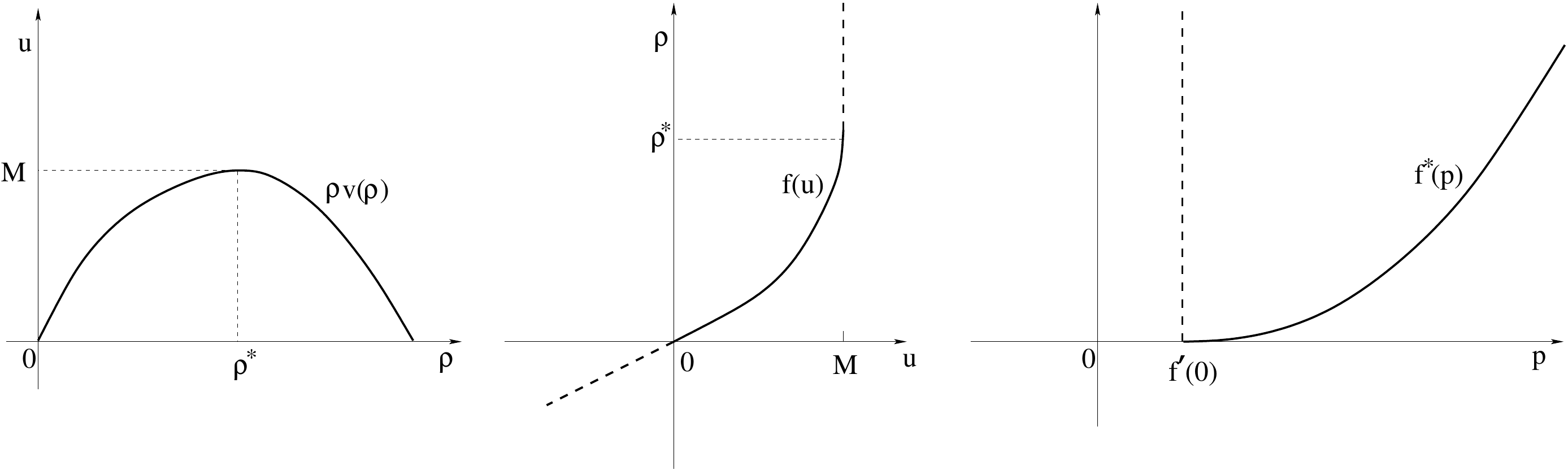}
   \caption{\small  Left: the function $\rho\mapsto
   \rho\,v(\rho)$ describing the flux of cars. Middle:
the  function $f$, implicitly defined
by $f(\rho v(\rho)) = \rho$ for $u\in [0,M]$, and then extended according to (\ref{fext}).
Right: the Legendre transform
   $f^*$ defined at (\ref{f*}).}
   \label{f:t44}
\end{figure}

The equation (\ref{claw1}) is supplemented by an initial condition
\bel{clic}
u(0,x)~=~\bar u(x),\eeq
where $\bar u(x)$ determines the flux of cars arriving at the entrance of the highway at time $x$.

Since the flux cannot exceed $M$, if the boundary data happens to be $\bar u(x)>M$ a queue is formed at the 
entrance of the highway. Cars enter at rate $M$ until the queue is depleted.
The actual solution is found by solving a Cauchy problem with an {\bf effective flux}
$\bar u^e$ defined as follows (see Fig.\ref{f:tf269}).
\bel{eflux}
\bar u^e(x)~\doteq~\left\{\bega{rl}
M\quad &\ds \hbox{if}\quad\int_y^x \bar u(s)\,ds > M(x-y)\quad\hbox{for some}~y<x,\\[2mm]
\bar u(x)&\hbox{otherwise.}\enda\right.\eeq
Notice that the first alternative occurs if the total amount of cars arriving at the entrance of the highway 
during the time interval $[y,x]$ is larger than the maximum possible entering flux. This is when a queue is formed.

\begin{figure}[htb]
   \centering
 \includegraphics[width=0.5\textwidth]{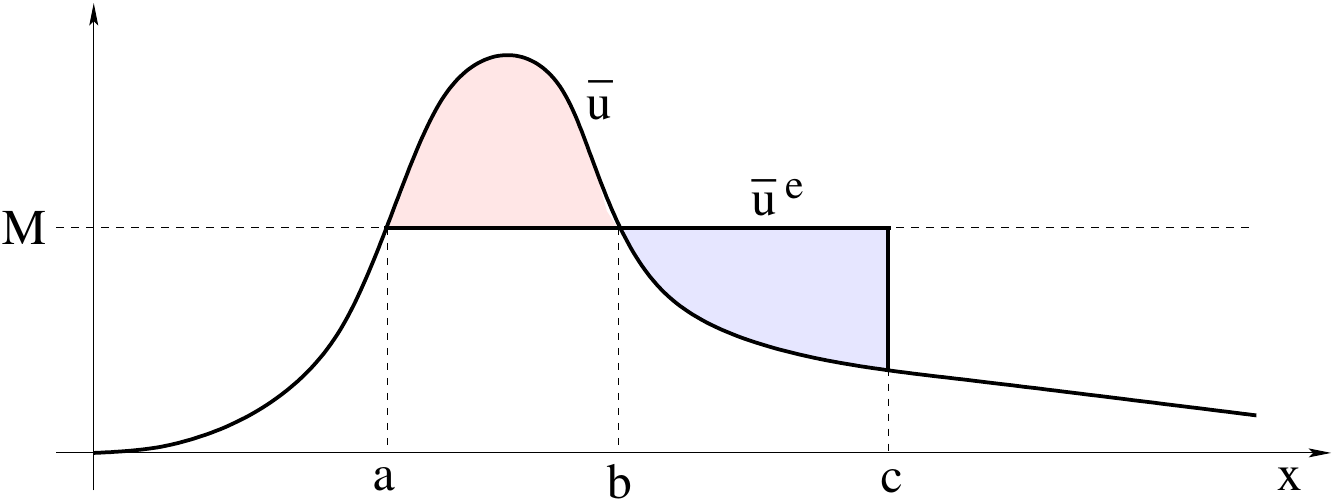}
   \caption{\small Here $\bar u$ is a departure rate, i.e.~the rate at which cars arrive at the entrance of the highway.   Since the flux cannot exceed $M$, the effective flux $\bar u^e$ is a function constantly equal to $M$
   on the interval $[a,c]$.  A queue at the entrance of the highway is formed at time $a$ and depleted at time $c$, 
   so that the two shaded regions have equal areas.}
   \label{f:tf269}
\end{figure}

\begin{lemma} \label{l:31} 
The map $\bar u\mapsto \bar u^e$ defined at (\ref{eflux}) is contractive w.r.t.~the $\L^1$ distance.
\end{lemma} 
{\bf Proof.} 
{\bf 1.} We claim  that  the map $\bar u\mapsto \bar u^e$ is monotone:
\bel{mono} \bar u(x)\leq \bar v(x) ~~\forall x\in\R\qquad\implies\qquad  \bar u^e(x)\leq \bar v^e(x) ~~\hbox{for a.e.~} x\in\R.\eeq
Indeed, assume $\bar u\leq \bar v$.   If
\bel{ibig}\int_y^x \bar u(s)\, ds~>~M(x-y)  \qquad\hbox{for some}~y<x,\eeq   then 
$$\int_y^x \bar v(s)\, ds~>~M(x-y)  \qquad\hbox{and hence}\qquad \bar u^e(x)\,=\, \bar v^e(x)\,=\,M.$$  
On the other hand, if $x$ is a Lebesgue point of both $\bar u$ and $\bar v$ and (\ref{ibig}) does not hold, then 
$\bar u^e(x)=\bar u(x)
\leq M$.  In this case,  
$$\bar u^e(x)~\leq~\min\{ \bar u(x), M\}~\leq~\min\{ \bar v(x),M\}~\leq~\bar v^e(x).$$
\v
{\bf 2.} Next, to prove $\L^1$-contractivity, consider the functions
$$\ov w_+(x)\,\doteq\,\max\bigl\{ \bar u(x), \bar v(x)\bigr\},\qquad\qquad 
\ov w_-(x)\,\doteq\,\min\bigl\{ \bar u(x), \bar v(x)\bigr\}.$$
By comparison,
$$\ov w_-^e(x)~\leq~\min\bigl\{ \bar u^e(x), \bar v^e(x)\bigr\}~\leq~\max\bigl\{ \bar u^e(x), \bar v^e(x)\bigr\}~\leq~\ov w_+^e(x)\qquad\quad \forall x\in\R.$$
Therefore
$$\bega{l} \ds \int_{-\infty}^{+\infty} \Big|\bar u^e(x) - \bar v^e(x)\Big|\, dx ~\leq~\int_{-\infty}^{+\infty} \Big(\ov w^e_+(x) - \ov w^e_-(x)\Big)\, dx \\[4mm]
\ds\quad =~
\ds \int_{-\infty}^{+\infty} \Big(\ov w_+(x) \, dx -\ov w_-(x)\Big) \, dx~=~ \int_{-\infty}^{+\infty} \Big|
\bar u(x) -\bar v(x)\Big| \, dx\,.
\enda
$$
\endproof

For $(t,x)\in [0,L]\times \R$, let  $u=u(t,x)$ be the entropy weak solution of the conservation law (\ref{claw1}),
with initial data
\bel{eid}u(0,x)~=~\bar u^e(x),\eeq
where the effective incoming flux was defined at (\ref{eflux}).
We denote by
$$\kappa~=~\int \bar u(x)\, dx~=~\int \bar u^e(x)\, dx$$ the total number of departing drivers. The Lagrangian variable $\beta\in [0,\kappa]$ will be used to 
label one particular driver.   For a.e.~$\beta\in [0,\kappa]$
the departure time and the arrival time of this driver are determined by
\bel{xda}x^d(\beta) ~=~ \inf \left\{ x\,;~\int_{-\infty} ^x \bar u(s)\, ds =\beta\right\},\qquad\qquad
x^a(\beta) ~=~ \inf \left\{ x\,;~\int_{-\infty} ^x  u(L,s)\, ds =\beta\right\}.\eeq
The difference
$$\ell(\beta)~=~x^a(\beta)-x^d(\beta)$$
is the duration of the journey of driver $\beta$.  For a given  $\kappa$, we 
have the a priori bounds
\bel{ellbd}
\ell_0~\leq~\ell(\beta)~\leq~\ell_{max}\,,\eeq
Here $\ell_0$ is the time needed to cover the distance $L$ at maximum speed, i.e.~when $\rho=0$ and 
the road is empty. Namely,
\bel{ell0}\ell_0~\doteq~ {\hbox{[length of the road]}\over \hbox{[maximum speed]}}~=~{L \over v(0)}~=~L f'(0). \eeq
To compute $\ell_{max}$, consider the worst possible scenario, where all drivers start their journey at the same time, say at $x=0$. Then a queue of size $\kappa$ is instantly formed at the entrance of the highway.
The corresponding effective flux is
\bel{worst}\bar u^e(x)~=~\left\{\bega{cl}  M\quad &\hbox{if} ~~x\in [0, \kappa/M],\\[2mm]
0\quad &\hbox{otherwise.}\enda\right.\eeq
With reference to Fig.~\ref{f:t44}, left, let $\rho^*$ be the density of cars corresponding to 
the maximum flux.
We then have the simple estimate
\bel{ellmax}
\ell_{max}~\leq~{\kappa\over M} + {L\over v(\rho^*)}\,.\eeq
Indeed, the last car needs to wait until time $x=\kappa/M$ to clear the queue.
After it starts moving,  since the traffic density is $\rho\leq \rho^*$, the velocity is $\geq v(\rho^*)$.
Hence any car needs at most time $L/v(\rho^*)$ to reach the end of the highway.

Given a departure distribution $\bar u$, the arrival time $y^{\bar u}(x)$ of a driver 
departing at time $x\in\R$ can be determined as
\bel{atime}y^{\bar u}(x)~\doteq~\inf\Big\{ y\geq x+\ell_0\,;~~\int_{-\infty}^y u(L,\xi)\, d\xi~\geq~\int_{-\infty}^x \bar u(\xi)\, d\xi\Big\}.\eeq
Notice that the expression (\ref{atime}) is meaningful also if $\bar u(x)=0$.  If the road is empty, the arrival time is $x+\ell_0$.  Otherwise, the arrival time is the earliest time when all cars starting before time $x$ arrive to destination.
The corresponding duration of the journey is 
\bel{ellbaru} \ell^{\bar u}(x)~\doteq~y^{\bar u}(x)-x.\eeq

\subsection{Representation of solutions by the Lax formula.}

As shown in Fig.~\ref{f:t44}, center,     it is convenient to extend $f$ to a function 
$f:\R\mapsto \R\cup \{+\infty\}$ by setting
\bel{fext}
f(u) ~\doteq~\left\{ \bega{cl} f'(0+) u
\qquad &\hbox{if}~~u<0,  
\cr
+\infty \qquad &\hbox{if}~~u>M.\enda\right.
\eeq
The Legendre transform of this extended flux function (see Fig.~\ref{f:t44}, right) will be denoted by
\bel{f*} f^*(p)~\doteq~ \max_u \bigl\{ pu-f(u)\bigr\}.\eeq
Since $f$ is convex, we can write
\bel{f*2}f^*(p)~=~p u(p) - f\bigl(u(p)\bigr),\eeq
where $u(p)$ is implicitly defined by
\bel{f'p}f'\bigl(u(p)\bigr)~=~p.\eeq
In turn, this implies
\bel{pff}
{\partial  f^*\over \partial p} ~=~u(p) +p\,{\partial u\over\partial p} - f'\bigl(u(p)\bigr) {\partial u\over\partial p} ~=~u(p),\eeq
\bel{ppff}{\partial^2 \over \partial p^2} f^*(p)~=~{\partial \over\partial p} u(p)~=~{1\over f''\bigl(u(p)\bigr)}. \eeq
Note that the last identity is obtained  differentiating (\ref{f'p}) and observing that
$$f''(u(p))~\cdot {\partial u\over\partial p} ~=~1.$$

Next, consider the integrated function
\bel{Udef}U(t,x)~\doteq~\int_{-\infty}^x u(t,s)\, ds\,.\eeq
By definition, $U(t,x) $ describes the total flux through the point $t$ on the highway
during the time interval $]-\infty, x]$.  Equivalently, this is the
 total number of cars that have crossed the point $t\in [0,L]$
before time $x$.

Integrating w.r.t.~$x$,
the conservation law (\ref{claw1}) can be 
equivalently written as a Hamilton-Jacobi equation~\cite{Ca}
\bel{hj}
U_t+ f(U_x)~=~0\,,\eeq
with initial data
\bel{Uin}
U(0,x)~=~\ov U(x)~\doteq~\int_{-\infty}^x \bar u(y)\, dy.\eeq

It is well known that $U=U(t,x)$ corresponds to the value function
for a variational problem. This fact is exploited by the 
Lax formula:
\bel{Lax}
U(t,x)~=~\inf\, \left\{
t\,f^*\left(\frac{x-z}{t}\right)+ \ov U(z)~~;~~~z\in\R
\right\}.\eeq
Knowing $U$, one can then recover the density $u(t,x)=U_x(t,x)$ at a.e.~point.
%

%
%

\begin{figure}[htb]
   \centering
 \includegraphics[width=0.8\textwidth]{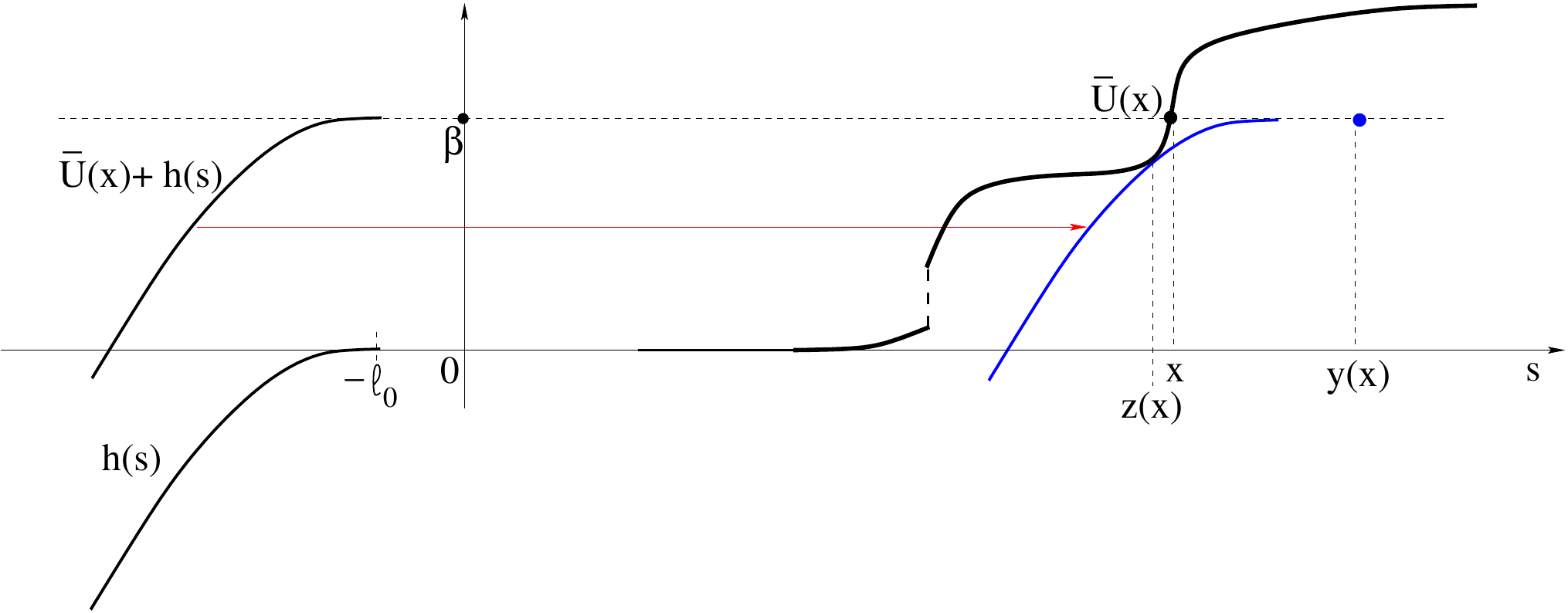}
   \caption{\small   Computing the arrival time $y(x)$ of a car that departs at time $x$.  Here $y(x)$ and $z(x)$
   are determined by the conditions in (\ref{neco}).}
   \label{f:t509}
\end{figure}

Based on the Lax formula, there is an appealing way to construct the arrival time
of each driver, shown in Fig.~\ref{f:t509}. 
Given the departure distribution $\ov U(\cdot)$ as in (\ref{Uin}), 
   let $\beta\in [0, \kappa]$
  label a particular driver. The departure time of this driver is 
\bel{bstart} x^d(\beta)~=~\min\bigl\{ x\,;~~\ov U(x)\geq \beta\bigr\},\eeq
while the arrival time is 
\bel{barrival} x^a(\beta)~=~\min\bigl\{ y\,;~~ U(L,y)\geq \beta\bigr\}.\eeq
Notice that
$$x^a(\beta)\leq y\qquad\iff\qquad U(L,y)\geq \beta.$$
In other words: driver $\beta$ arrives at some time $\leq y$  if and only if at time $y$ the 
number of drivers that have crossed the finish line at $L$ is $\geq \beta$.
Using  (\ref{Lax}) with $t=L$, we find that this is true iff
\bel{atime2}\beta~\leq~U(L,y)~=~\inf\, \left\{
L\,f^*\left(\frac{y-z}{L}\right)+ \ov U(z)~~;~~~z\in\R
\right\},\eeq
where $f^*$ is the Legendre transform of $f$.  
Hence the arrival time is
\bel{arrtime}\bega{rl}
x^a(\beta)&\ds=~\inf\left\{ y\,;~~L\,f^*\left(\frac{y-z}{L}\right)+ \ov U(z)\geq\beta\quad\forall z\in\R\right\}
\\[4mm]
&\ds=~\inf \left\{ y\,;~~\beta - L\,f^*\left(\frac{y-z}{L}\right)\leq \ov U(z)\quad\forall z\in\R\right\}
.\enda\eeq
Consider the function (see again Fig.~\ref{f:t509})
\bel{hdef}h(s)~\doteq~- Lf^*(-s/L).\eeq
Note that 
$$h(s)\nearrow 0\qquad \hbox{as}\qquad s\nearrow -\ell_0~\doteq~-L f'(0).$$

We can now write (\ref{arrtime}) as
\bel{taf1}x^a(\beta)~=~\inf\Big\{ y\,;~\beta + h(z-y)\leq \ov U(z)~~\hbox{for every} ~z\in\R\Big\}.\eeq
As shown in Fig.~\ref{f:t509}, one needs to
shift
the graph of the function $x\mapsto \beta+h(x)$ horizontally, until
it lies entirely below the graph of $\ov U$.  In the figure, the size
of this shift yields precisely the arrival time $y(x)=x^a(\beta)$. 

In the smooth case, if $x=x^d(\beta)$  is the departure time of driver $\beta$, then
the arrival time $y(x)=x^a(\beta)$ and the point $z(x)$, where equality in (\ref{taf1}) is achieved, are determined by the two identities (necessary conditions for
minimality)
\bel{neco} \beta + h(z-y) \,=\,\ov U(z),\qquad\qquad h'(z-y) \,=\,\ov U'(z).\eeq

For future applications, we shall need to compute the arrival time $y^{\bar u}(x)$ of a driver departing at time $x$,
also in the case where $\bar u(x)=0$.    This is achieved by the formula
\bel{artim} y^{\bar u}(x)~=~\inf\Big\{ y\geq x+\ell_0\,;~~\ov U(x)  + h(z-y)\leq \ov U(z)~~
\hbox{for every} ~z\le x\Big\}.\eeq

\begin{figure}[htb]
   \centering
 \includegraphics[width=0.4\textwidth]{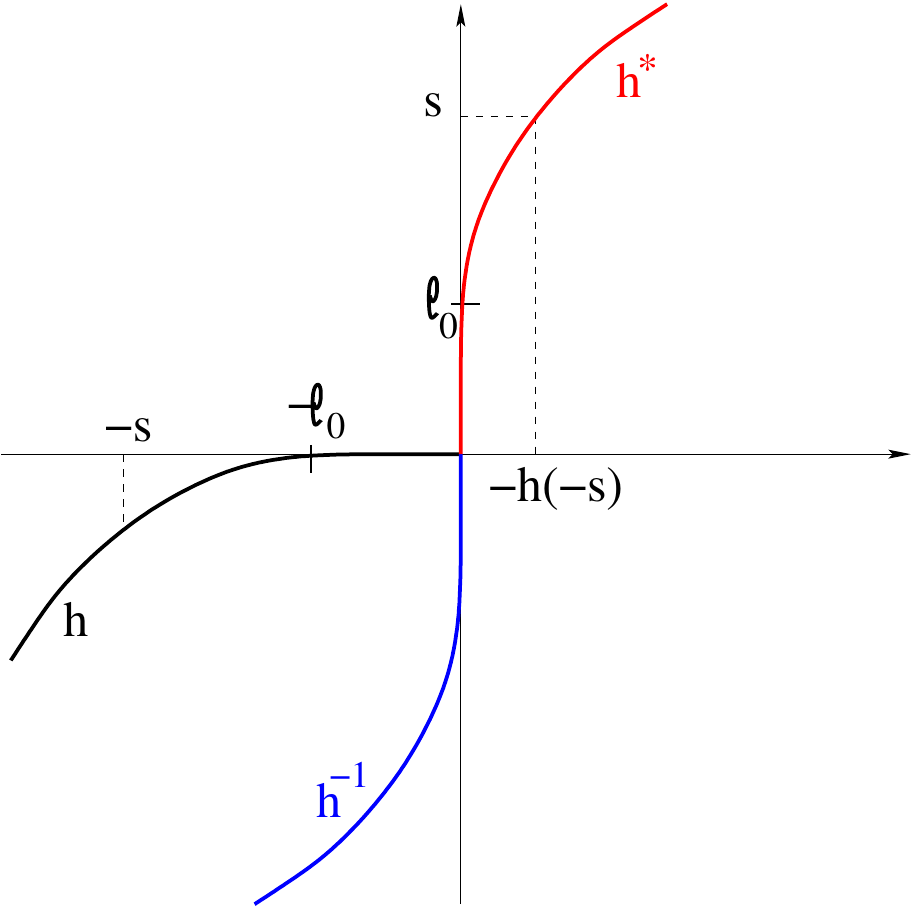}
   \caption{\small   The functions $h$ and $h^*$ defined at (\ref{hdef}) and (\ref{h*}) respectively.}
   \label{f:h99}
\end{figure}

\begin{remark} {\rm The arrival time $x^a(\beta)$ can be also obtained as a sup-convolution.
Namely, consider the inverse function (see Fig.~\ref{f:h99})
\bel{h*}h^*(\zeta)~\doteq~ -h^{-1}(-\zeta), \eeq
so that
\bel{h3}h^*\bigl( - h(-s)\bigr)~=~s\qquad\qquad\forall s\geq 0.\eeq
We then have
\bel{supco}x^a(\beta)\,=\,\sup\Big\{ x^d(\xi) +h^*(\beta-\xi)\,;~~\xi\in[0, \beta]\Big\}.\eeq 
Indeed, in view of (\ref{taf1}) we have the equivalences
$$\bega{rll} y\geq x^a(\beta)\quad &\iff\quad  h(z-y)\,\leq\, \ov U(z)-\beta~~&\hbox{for every} ~z\leq y
\\[2mm]
&\iff\quad  - h(z-y)\,\geq\, \beta -\ov U(z)~~&\hbox{for every} ~z\leq y
\\[2mm]
&\iff\quad  y-z~\geq~ h^*\bigl( \beta- \ov U(z)\bigr)~~&\hbox{for every} ~z\leq y\\[2mm]
&\iff\quad  y~\geq~z+ h^*\bigl( \beta- \ov U(z)\bigr)~~&\hbox{for every} ~z\leq y
\\[2mm]
&\iff\quad  y~\geq~x^d(\xi)+ h^*\bigl( \beta- \xi\bigr)~~&\hbox{for every} ~\xi\leq \beta.
\enda$$
This yields (\ref{supco}).
}
\end{remark}

\section{H\"older continuous dependence of the arrival time function in the Lagrangian formulation}
\label{sec:3}
\setcounter{equation}{0}
As before, we let $\kappa$ be the total number of drivers, while $L$ is the length of the road.
The Lagrangian variable $\beta\in [0, \kappa]$ labels one particular driver.
Calling $x^d(\beta)$ and $x^a(\beta)$ respectively the departure time and the arrival time of
driver $\beta$, we wish to understand the continuity properties of the map $x^d(\cdot)\mapsto x^a(\cdot)$, as a map from $\L^1\bigl([0,\kappa]\bigr)$ into itself.

\begin{theorem}\label{t:31}
Let the flux function $f$ satisfy {\bf (A1)}.  
There exists a constant $\Tilde C$, depending only on $f$ and on the constants $\kappa, L,B>0$, such that the following holds.

For any two nondecreasing departure time functions $x_1^d,x_2^d:[0,\kappa]\mapsto [-B,B]$,
the arrival times $x^a_1,x_2^a$ satisfy
\bel{Hoart}\|x^a_1-x^a_2\|_{\L^1}~ \le~ \Tilde C \|x^d_1-x^d_2\|_{\L^1}^{1/2}.\eeq
\end{theorem}

In view of the formula (\ref{supco}) for the arrival time function, we first prove a general lemma on the H\"older continuity 
of the sup-convolution operator.

Given constants $\kappa, B$,
let $\F$ be the family of all nondecreasing functions $\phi:[0,\kappa]\mapsto [-B,B]$.
For $\phi\in \F$, consider the function $\psi = \Psi(\phi)$, defined by
\bel{psi1}
\psi(x)~=~\sup_{0\leq y\leq x} \phi(y)+\ell(x-y).\eeq
The following properties follow immediately from the definition:

{\bf 1 - Comparison:}
\bel{comp}\phi_1~\leq~\phi_2\qquad\implies\qquad\psi_1\leq\psi_2\,.\eeq

{\bf 2 - Contractivity in $\L^\infty\bigl([0,\kappa]\bigr)$:}
\bel{cont} \|\psi_1-\psi_2\|_{\L^\infty}~\leq~\|\phi_1-\phi_2\|_{\L^\infty}\,.\eeq
\v
Next, we seek continuity estimate of the map $\Psi:\phi\mapsto \psi$ in terms of the $\L^1$ 
distance.
\v
\begin{example}\label{e:11} {\rm  To see that the map $\phi\mapsto\psi$ may not be Lipschitz continuous 
from
$\L^1\bigl([0,\kappa]\bigr)$ into itself, take $\kappa=1$, $\ell(s) = \sqrt s$. As shown in Fig.~\ref{f:h97},
consider the two functions
\bel{ex1}\phi_1(x)\,=\,0,\qquad\phi_{2}(x)~\doteq~\left\{\bega{cl}
-\ve\quad &\hbox{if}~~x\in [0, 2\ve[\,,\\[1mm]
0\quad &\hbox{if}~~x\in [2\ve,1].
\enda\right.\eeq
We then have
$$\psi_1(x)\,=\,\sqrt x\,,~~\qquad \psi_2(x)\,=\, \max\bigl\{ \sqrt x-\ve, ~\sqrt{x-2\ve}\bigr\}\,=\,\sqrt x - \ve
\qquad \forall x\in [0,1].$$
Hence
$$\|\phi_1-\phi_2\|_{\L^1}\,=\,2\ve^2,\qquad\qquad\|\psi_1-\psi_2\|_{\L^1}\,=~\ve.$$
}
\end{example}
\v

\begin{figure}[ht]
\centerline{\hbox{\includegraphics[width=4cm]{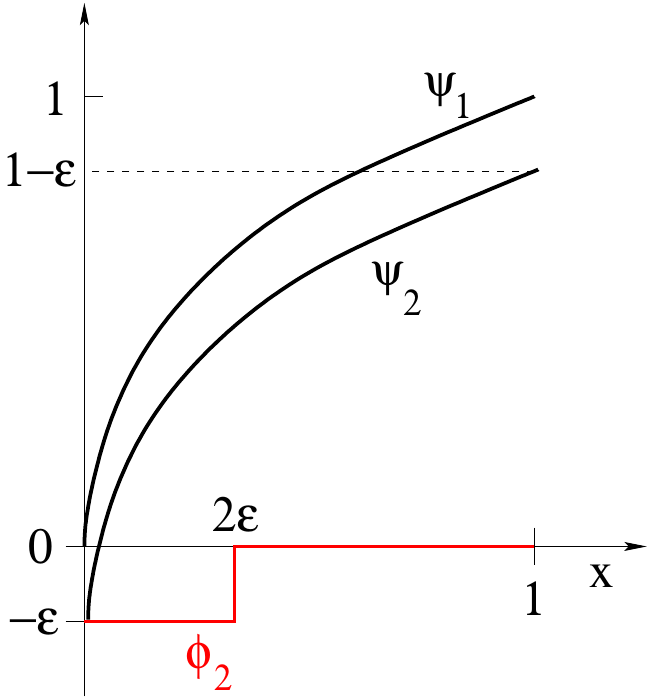}}}
\caption{\small The functions $\phi_1=0, \phi_2$ and $\psi_1,\psi_2$, considered in Example~\ref{e:11}. }
\label{f:h97}
\end{figure}

The following lemma shows the map $\Psi:\phi\mapsto \psi$ is in fact H\"older continuous with exponent $\alpha=1/2$.

\begin{lemma}\label{p:LagrHolder}
Let $\kappa,A,B>0$ and $\ell:[0,\kappa] \mapsto [0,A]$ and  $\phi_1,\phi_2:[0,\kappa]\mapsto[-B,B]$ be nondecreasing functions. Define the corresponding sup-convolutions $\psi_1,\psi_2$  as in (\ref{psi1}).

Then 
\bel{21}
\|\psi_1-\psi_2\|_{\L^1}^2 ~\le ~4\kappa(A+4B) \|\phi_1-\phi_2\|_{\L^1}.
\eeq
\end{lemma}

{\bf Proof.}
By the comparison property (\ref{comp}) it is not restrictive to assume $\phi_2\le \phi_1$. 

Since $\phi_1$ and $\phi_2$ are nondecreasing, for $0<\delta\le\kappa$ and $s+\delta\le\kappa$, it holds
\[
\phi_1(s)-\phi_2(s+\delta)\le \frac1\delta \int_s^{s+\delta}(\phi_1(x)-\phi_2(x))dx \le \frac1\delta \|\phi_1-\phi_2\|_{\L^1}.
\]

Then for $x\le\kappa-\delta$, it holds
\begin{equation*}
\begin{split}
\psi_1(x) \le &~ \sup_{0\le y\le x} \{ \phi_2(y+\delta) + \ell(x-y) \} + \frac1\delta \|\phi_1-\phi_2\|_{\L^1} \\
\le &~ \sup_{\delta\le y\le x+\delta} \{ \phi_2(y) + \ell(x+\delta-y) \}+ \frac1\delta \|\phi_1-\phi_2\|_{\L^1} \\
\le &~ \psi_2(x+\delta)+ \frac1\delta \|\phi_1-\phi_2\|_{\L^1}.
\end{split}
\end{equation*}

The function $\psi_2$ is nondecreasing and has oscillation at most
$A+2B$. Also $0\le\psi_1-\psi_2\le 2B$, by sup-norm contractivity. Therefore
\begin{align*}
\int_0^\kappa(\psi_1-\psi_2)\,dx
&\le\int_0^{\kappa-\delta}\bigl[\psi_2(x+\delta)-\psi_2(x)\bigr]\,dx
+\frac{\kappa }{\delta}\|\phi_1-\phi_2\|_{\L^1}+2B\delta\\
&\le (A+4B)\delta+\frac{\kappa }{\delta}\|\phi_1-\phi_2\|_{\L^1}.
\end{align*}
Let $\delta=\sqrt{\kappa\|\phi_1-\phi_2\|_{\L^1}/(A+4B)}$. It holds $\delta \le \kappa$ because $\|\phi_1-\phi_2\|_{\L^1}\le 2B\kappa$ and therefore the result follows. 
\endproof

\v
{\bf Proof of Theorem~\ref{t:31}.} In view of the sup-convolution formula (\ref{supco}) for the arrival time, 
it suffices to apply Lemma~\ref{p:LagrHolder} with $\ell (s) \doteq h^*(s)\doteq -h^{-1}(-s)$ (see Fig.~\ref{f:h99}).
\endproof

\v
\section{Lipschitz continuous dependence of the arrival time function in the Eulerian formulation}
\label{sec:4}
\setcounter{equation}{0}
For future use, in this section we prove the Lipschitz continuous dependence  of the arrival time  in terms of the 
departure rate.  Namely  the map $\bar u(\cdot)\mapsto y^{\bar u}(\cdot)$ in (\ref{artim}) is Lipschitz continuous w.r.t.~the $\L^1$ distance.  Note that here we look at the rates of departures and arrivals
as functions of $x\in \R$.

\begin{theorem} \label{t:41} 
Given any $K>0$, there exists a Lipschitz constant $L_K$, depending only on $K$
and on the flux function $f$, such that
\bel{Liparr}
\bigl\| y^{\bar u} - y^{\bar v}\bigr\|_{\L^1(\R)} 
~\leq~L_K \bigl\| \bar u - \bar v\bigr\|_{\L^1(\R)} \eeq
for every couple of non-negative departure rates such that $\|\bar u\|_{\L^1}\leq K$, $\|\bar v\|_{\L^1}\leq K$.
\end{theorem} 
{\bf Proof.} 
{\bf 1.} 
By a density argument it suffices to prove the inequality (\ref{Liparr}) in the case where
\begi
\item[(i)]
$\bar u, \bar v$ are both piecewise constant, say with jumps at the points $x_0<x_1<\cdots< x_n$. Namely,
\bel{ubi}
\bar u(x)\,=\,u_i,\qquad \bar v(x)\,=\, v_i\qquad\quad \hbox{for} ~x_{i-1}<x< x_i\,.\eeq
\item[(ii)]  
$\bar v(x)=\bar u(x)$ for all $x\notin [a,b]$, where $[a,b]=[x_{k-1}, x_k]$ is one of the intervals where $\bar u $ and $\bar v$ are constant.  For $x\in [a,b]$, 
the constant values of  $\bar u(x)$ and $\bar v(x)$ will be denoted by 
$ \sigma_u\doteq u_k<\sigma_v\doteq v_k$.   
\endi
To justify this reduction we observe that, given any two piecewise constant functions $\bar u, \bar v$, one can construct a chain of finitely many
piecewise constant functions $\bar u_j$, $j=0,1,\ldots, N$ such that (see Fig.~\ref{f:t513})

\begin{figure}[htb]
   \centering
 \includegraphics[width=0.5\textwidth]{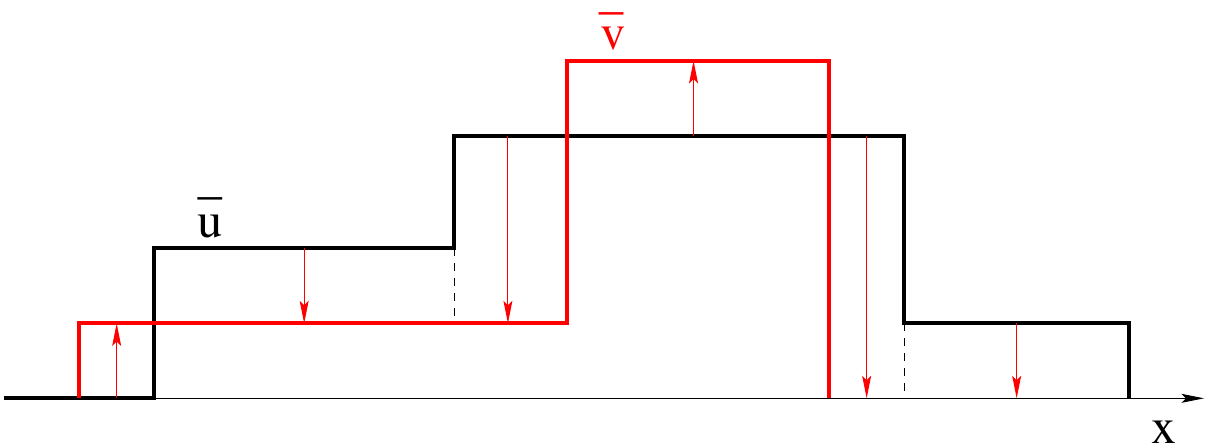}
   \caption{\small The departure rate $\bar u$ can be transformed into $\bar v$, in 6 elementary steps.}
   \label{f:t513}
\end{figure}

\begi
\item $\bar u_{j-1} (x)=\bar u_j(x)$ for all $x\notin [a,b]$, where $[a,b]$ is an interval where $\bar u_{j} $ is constant.  
\item We have
\bel{uj1} \bar u_0= \bar u, \qquad \bar u_N = \bar v,\qquad \| \bar v - \bar u\|_{\L^1}~=~\sum_{j=1}^N \|\bar u_j-\bar u_{j-1}\|_{\L^1}\,.\eeq
\endi
This reduces the analysis to the basic case (i)-(ii) above, as shown in Fig.~\ref{f:t507}. 
The corresponding integrated functions $\ov U $ and $\ov V$ are shown 
in Figures~\ref{f:t518} and \ref{f:t519}.  In this case, the inequality (\ref{Liparr}) reduces to
\bel{eyLip}  \bigl\| y^{\bar u} - y^{\bar v}\bigr\|_{\L^1(\R)} ~=~\O(1)\cdot\delta,
\qquad \quad \delta\,\doteq\, \|\bar u-\bar v\|_{\L^1}\,=\,(b-a)(\sigma_v-\sigma_u).
\eeq

\begin{figure}[htb]
   \centering
 \includegraphics[width=0.6\textwidth]{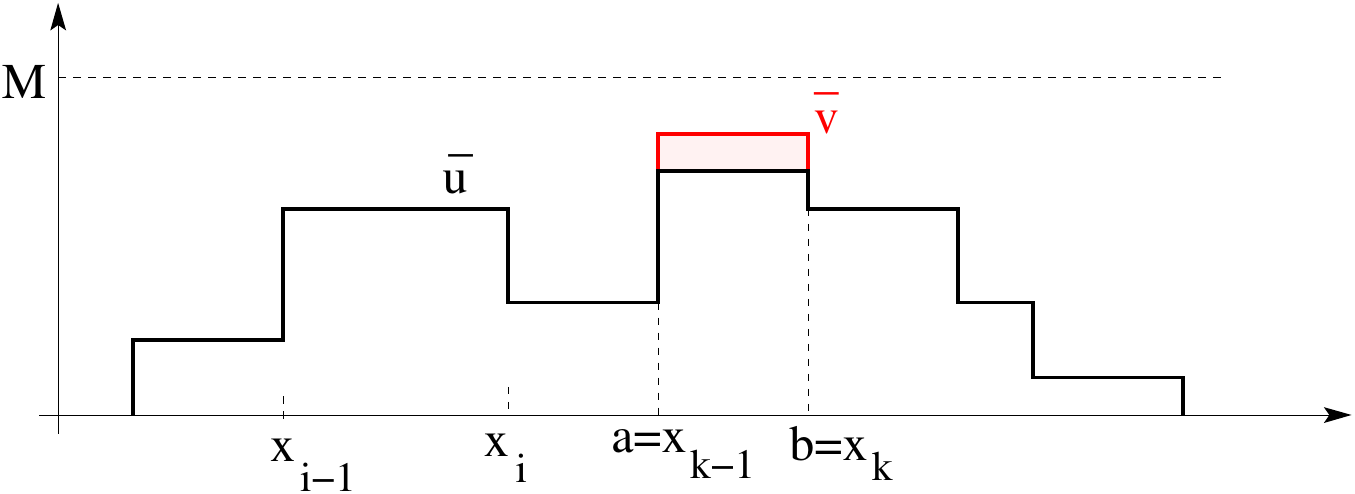}
   \caption{\small A piecewise constant departure rate $\bar u$ and a perturbation $\bar v$ of basic type.}
   \label{f:t507}
\end{figure}

\begin{figure}[htb]
   \centering
 \includegraphics[width=0.8\textwidth]{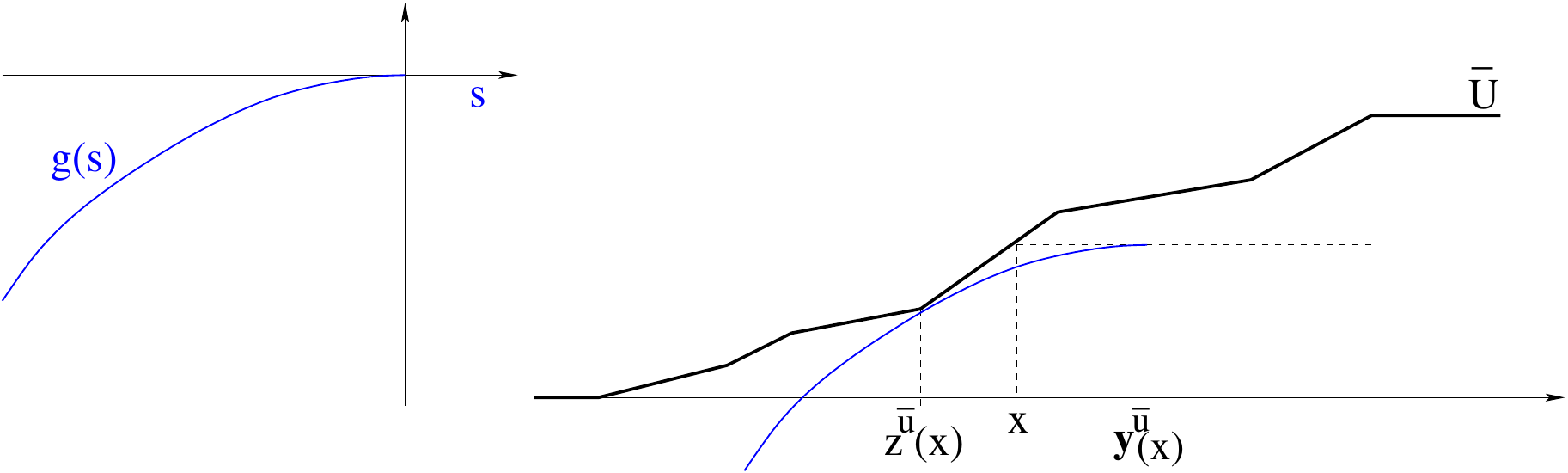}
   \caption{\small Left: the function $g$ at (\ref{gdef}).  Right:  the point $\bfy^{\bar u}(x)$ and the contact point $z^{\bar u}(x)$, introduced at (\ref{bfyx})-(\ref{zbx}). }
   \label{f:t514}
\end{figure}

\v
{\bf 2.} To simplify some of the computations, it will be convenient to 
introduce the points 
$$\bfy^{\bar u}(x)~\doteq~y^{\bar u}(x)- \ell_0\,.$$
 Since we trivially  have
$$\bfy^{\bar v}(x)-\bfy^{\bar u}(x)~=~y^{\bar v}(x)- y^{\bar u}(x),$$
we shall focus on the change in the values of $\bfy(x)$, which is easier to analyze. 
Consider the auxiliary function (see Fig.~\ref{f:t514}, left)
\bel{gdef} g(s)~\doteq~h(s-\ell_0)~=\,- L \,f^*\left( \ell_0-s\over L\right)~=\,-L f^*\left(f'(0)- {s\over L}\right),\qquad\qquad s\leq 0.\eeq
This function satisfies
\bel{gp1}g(0) = g'(0)=0, \qquad g''(0) <0, \qquad\lim_{s\to -\infty} g'(s) ~=~\lim_{p\to +\infty} {\partial \over\partial p} f^*(p)
~=~\lim_{p\to +\infty} u(p)~=~M.\eeq
Moreover, given $\kappa>0$, we can find $c_\kappa>0$ such that
\bel{gp2}g(s)\in [-\kappa, 0]\qquad\implies\qquad g(s)~\leq~- c_\kappa s^2.\eeq

Let $\ov U, \ov V$ be the integrated functions, as in (\ref{Uin}).
As shown in Fig.\ref{f:t514}, right,  given a departure rate $\bar u(\cdot)$, the point $\bfy^{\bar u}(x)$ can be computed 
by the formula
\bel{bfyx}
\bfy^{\bar u}(x)~=~\inf\Big\{ y\geq x\,;~~\ov U(s)\geq \ov U(x) + g(s-y)\quad\forall s\leq x\Big\}.\eeq
In other words, we consider the graph of the function $\ov U(x) + g(\cdot -y)$, and slide it forward 
(i.e., to the right)  until it
lies entirely below the graph of $\ov U(\cdot)$.
The contact point between the graphs of these two functions will be denoted by
\bel{zbx}
z^{\bar u}(x)~=~\max\Big\{ s\leq x\,;~\ov U(s)=\ov U(x) + g\bigl( s- \bfy^{\bar u}(x)\Big\}.\eeq
Notice that we have the implication
\bel{yuv<}
\bar u\leq \bar v\qquad\implies  \qquad \bfy^{\bar u}(x)\,\leq\,\bfy^{\bar v}(x)
\quad\forall x\in \R.\eeq
Indeed, if $ \ov V(s)\geq \ov V(x) +g\bigl(s-\bfy^{\bar v}(x)\bigr)$ for all $s<x$, then 
$$\ov U(s)~=~\ov V(s) + \bigl( \ov U(s)- \ov V(s)\bigr)~\geq~
\ov V(s) + \bigl( \ov U(x)- \ov V(x)\bigr)~\geq~
\ov U(x) + g\bigl( s- \bfy^{\bar v}(x)\bigr)$$
for all   $s<x$.  By the definition (\ref{bfyx}) this implies $\bfy^{\bar u}(x)\,\leq\,\bfy^{\bar v}(x)$.

%
%
%
\v
{\bf 3.} We are now ready to estimate the $\L^1$ distance in  (\ref{Liparr}), where
$\bar u, \bar v$ are piecewise constant functions with the properties listed in (i)-(ii).
Consider the point  $b^*\doteq \sup \{ x : z^{\bar v}(x)<b\}$.
For $x \in \R \setminus [a,b^*]$ one has  $\bfy^{\bar v} (x) = \bfy^{\bar u}(x)$. 
Therefore it is sufficient to consider $x \in [a, b^*]$.

Define the points
\bel{a*b'}a^*\doteq \inf \{ x: z^{\bar v}(x)>a\},
\qquad\qquad b' \doteq \sup \{ x : \bfy^{\bar v}(x)<b\}.\eeq
Since both cases  $a^*\le b'$ and $ b' < a^*$ are possible, in the next steps
we estimate  $\|\bfy^{\bar v}-\bfy^{\bar u}\|_{\L^1}$ separately on the interval $[a,a^*]$, 
on the possibly empty interval $[a^*,b']$, and on the interval $\bigl[\max\{a^*,b'\},\,b^* \bigr]$.

\begin{figure}[htb]
   \centering
 \includegraphics[width=1.0\textwidth]{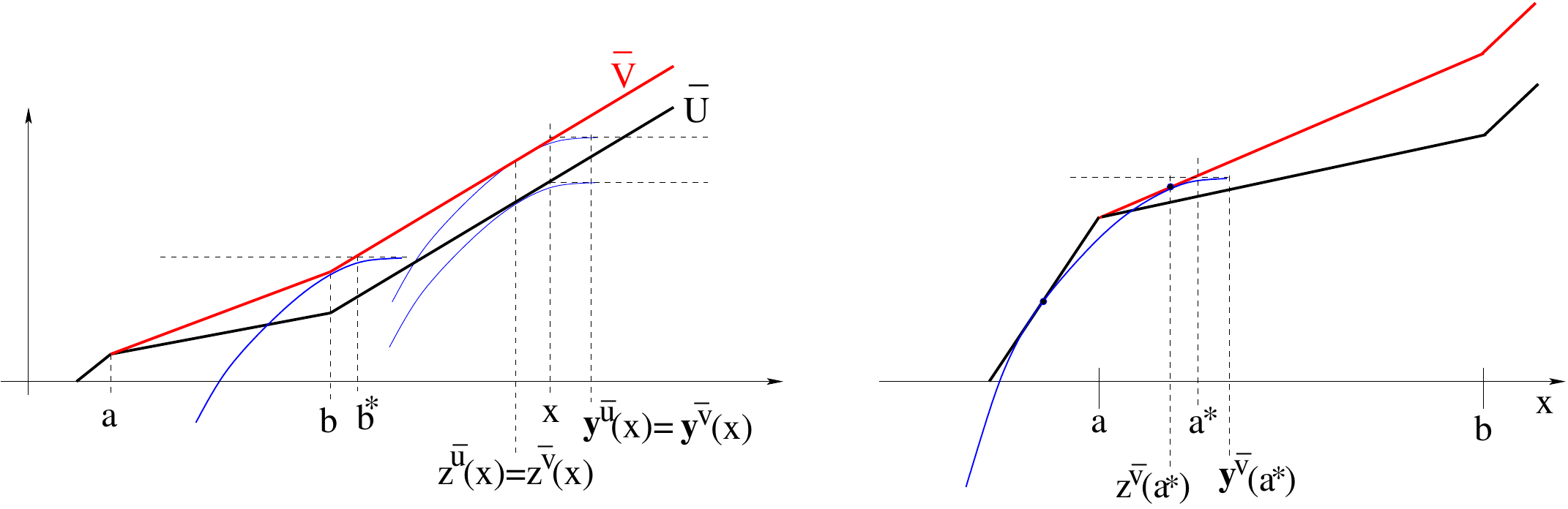}
   \caption{\small  Left: for $x>b^*$ one has $z^{\bar u}(x) = z^{\bar v}(x)$ and 
   $\bfy^{\bar u}(x) = \bfy^{\bar v}(x)$.  Right: an example of the point $a^*$ defined at  
   (\ref{a*b'}).}
   \label{f:t518}
\end{figure}

\begin{figure}[htb]
   \centering
 \includegraphics[width=0.7\textwidth]{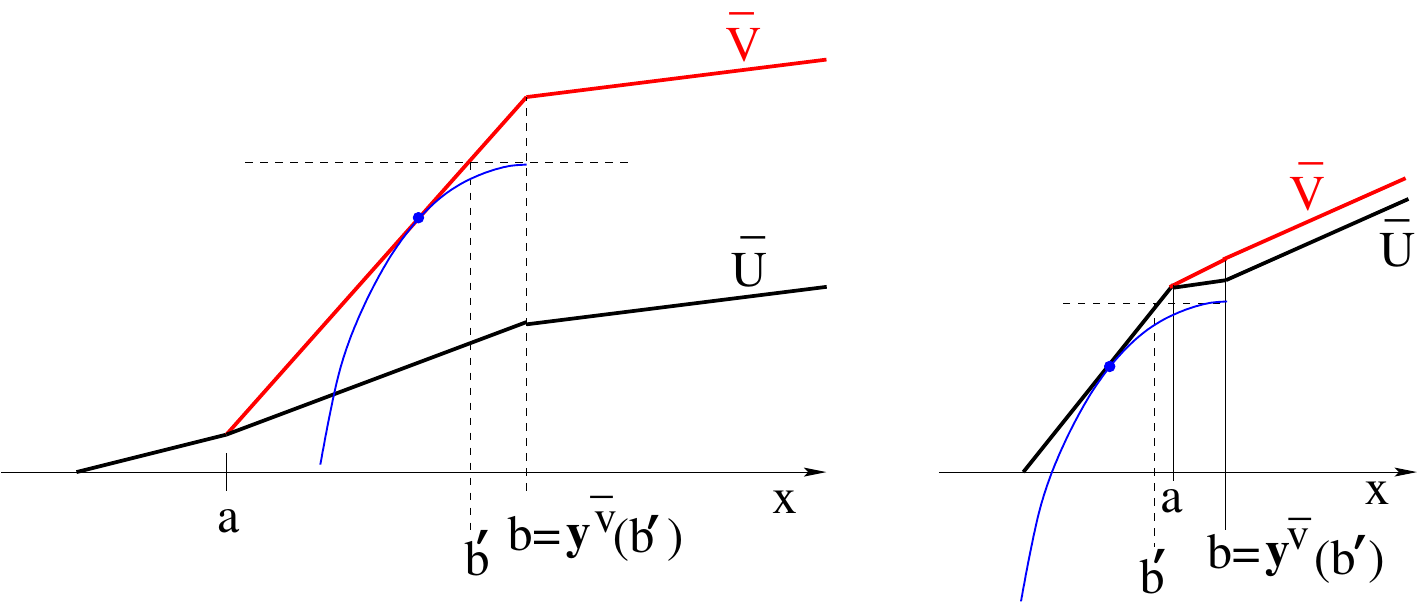}
   \caption{\small  Left: an example where $b'>a$. Right: an example where $b'<a$.}
   \label{f:t519}
\end{figure}

\v
{\bf 4.} Let $C_K>0$ be a constant such that, for every $\bar u \in \L^1(\R)$ with $\|\bar u\|_{\L^1(\R)}\le K$, there holds  
\[
\bfy^{\bar u}(x) - z^{\bar u}(x)\,\in \,[0, C_K]\qquad \forall x \in \R.
\]
Then choose a constant $c_K>0$ such that
\bel{cK}
g''(s) \le - c_K \qquad \forall s \in [-C_K,0].
\eeq
We claim that, for every $x \in \R$ such that $z^{\bar v}(x)<x$, one has
\bel{eLip}
\bfy^{\bar v}(x)-\bfy^{\bar u}(x) ~\le ~\frac{2}{c_K(\bfy^{\bar v}(x)-z^{\bar v}(x))}\, \delta.
\eeq
Indeed, since 
$$
\ov V(x) + g(z^{\bar v}(x)- \bfy^{\bar v(x)}) \ge \ov U(x) + g(z^{\bar v}(x)- \bfy^{\bar u(x)}) 
$$
then
\[
\bega{rl}
\delta& \ge ~\ov V(x) - \ov U(x) ~\ge \ds~ \int_{z^{\bar v}(x)}^{z^{\bar v}(x)+ \bfy^{\bar v}(x) - \bfy^{\bar u}(x)}g'\bigl(s-\bfy^{\bar v}(x)\bigr)ds ~\ge ~\bigl( \bfy^{\bar v}(x) - \bfy^{\bar u}(x)\bigr) \frac{g\bigl(z^{\bar v(x)}- \bfy^{\bar v}(x) \bigr)}{z^{\bar v}(x)-\bfy^{\bar v}(x)} \\[4mm]
 &\ge~ \bigl( \bfy^{\bar v}(x) - \bfy^{\bar u}(x)\bigr) \frac{c_K}2 \bigl(\bfy^{\bar v}(x)-z^{\bar v}(x)
 \bigr).
\enda
\]

\begin{figure}[htb]
   \centering
 \includegraphics[width=0.6\textwidth]{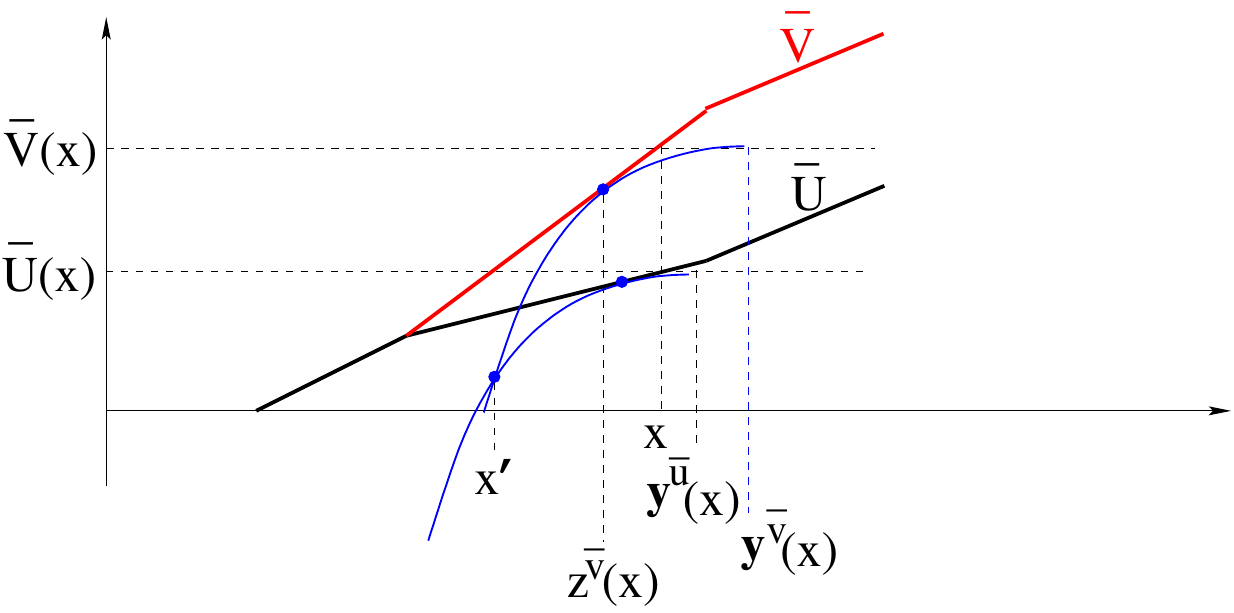}
   \caption{\small The points $x'<z^{\ov v}(x)$ and $\bfy^{\bar u}(x)<\bfy^{\bar v}(x)$ used in the proof of (\ref{eLip}).}
   \label{f:tn1}
\end{figure}

\v
{\bf 5.} 
 {\bf Estimate on $[a,a^*]$}.
 We define 
 $$\sigma^*\,\doteq \, g'\bigl(a-\bfy^{\bar v}(a^*)\bigr),\qquad\quad V^*\,\doteq \,\ov V(a^*)+ g(a-\bfy^{\bar v}(a^*))\le \ov V(a).$$ 
 Observe that $a^*-a \le \bfy^{\bar v}(a^*) -a \le {\sigma^* / c_K}$.

By the definition of $a^*$ at  (\ref{a*b'}) it follows that, for every $ x \in \,]a,a^*[\,$, 
there holds
\[
\ov V(x) + g(a-\bfy^{\bar v}(x))~\le~ V^*.
\]
 Moreover for every $x \in \,]a,a^*[$ it holds 
 \[
 \bega{rl}
 \ov V(x) &\ge ~ \ov V(a^*) + g(x-\bfy^{\bar v}(a^*)) ~
 = ~ V^* + g(x-\bfy^{\bar v}(a^*)) - g(a-\bfy^{\bar v}(a^*)) \\[3mm]
 &\ds \ge ~ V^* + (x-a) {g(a-\bfy^{\bar v}(a^*)) \over a-\bfy^{\bar v}(a^*)} ~\ge ~ V^* + (x-a) \frac{\sigma^* c_K}{2 \|g''\|_{\L^\infty}}\,.
 \enda
 \]
 Therefore
  \bel{e11}
 \bega{l}\ds
 \int_a^{a^*}\frac1{\bfy^{\bar v}(x)-z^{\bar v}(x)}dx ~\le ~ \int_a^{a^*}\frac1{\bfy^{\bar v}(x)-a}
 \,dx~\le~ \sqrt{\frac{\|g''\|_{\L^\infty}}{2} } \int_a^{a^*}\frac1{\sqrt{-g\bigl(a-\bfy^{\bar v}(x)\bigr)}}\,dx\\[4mm]
 \quad\ds \le ~ \sqrt{\frac{\|g''\|_{\L^\infty}}{2}} \int_a^{a^*}\frac1{\sqrt{\ov V(x)-V^*}}\,dx 
 ~\le ~\frac{\|g''\|_{\L^\infty}^{1/2}}{\sqrt{\sigma^* c_K}} \int_a^{a^*} \frac1{\sqrt{x-a}}\,dx ~
 \le ~  {2\|g''\|_{\L^\infty}\over c_K} \,.
 \enda
 \eeq
 In view of (\ref{eLip}), this yields
\bel{e12} \int_a^{a^*}\bigl| \bfy^{\bar v}(x) - \bfy^{\bar u}(x)\bigr|\, dx ~\le~ {2\delta\over c_K}  \int_a^{a^*}\frac1{\bfy^{\bar v}(x)-z^{\bar v}(x)}dx ~\leq~ \frac{4\|g''\|_{\L^\infty}}{c_K^2} \cdot \delta.
\eeq
 \v
 {\bf 6.} {\bf Estimate on $[a^*, b']$}. This case is nontrivial only if $b' > a^*$. We claim that 
 \bel{e13}
 \bfy^{\bar v}(x)-\bfy^{\bar u}(x)~\le~ \frac1{c_K}\left(1+ \frac{\|g''\|^2_{\L^\infty}}{c_K^2}\right)(\sigma_v-\sigma_u)\qquad \quad \forall ~x \in \ ]a^*,b'[\ \subset [a,b].
 \eeq 
 To prove \eqref{e13}, consider any $x \in \ ]a^*, b' [ \,$. Then $a < z^{\bar v}(x) < \bfy^{\bar v}(x) < b$, hence $\bfy^{\bar v}(x)-z^{\bar v}(x) = \Delta_{\sigma_v}$ with $g'(-\Delta_{\sigma_v})=\sigma_v$ and $-g(-\Delta_{\sigma_v})= \sigma_v (x- z^{\bar v}(x))$.
 Therefore
 \[
 \bfy^{\bar v}(x) \,=\, x + \Delta_{\sigma_v} + \frac{g(-\Delta_{\sigma_v})}{\sigma_v}.
 \] 
 Similarly we obtain 
 \[
  \bfy^{\bar u}(x) \,\ge\, x + \Delta_{\sigma_u} + \frac{g(-\Delta_{\sigma_u})}{\sigma_u}.
 \]
 In particular 
  \[
 \begin{split}
 \bfy^{\bar v}(x) - \bfy^{\bar u}(x)  \le  &~  \Delta_{\sigma_v} + \frac{g(-\Delta_{\sigma_v})}{\sigma_v} - \Delta_{\sigma_u} - \frac{g(-\Delta_{\sigma_u})}{\sigma_u} \\
 =&~  \Delta_{\sigma_v} - \Delta_{\sigma_u} + \frac{g(-\Delta_{\sigma_v})}{g'(-\Delta_{\sigma_v})}  - \frac{g(-\Delta_{\sigma_u})}{g'(-\Delta_{\sigma_u})} \\
 \le &~ \left(1+ \left\|\frac{g'' g}{(g')^2}\right\|_{\L^\infty([-C_K,0])}\right)(\Delta_{\sigma_v}-\Delta_{\sigma_u}) \\
\le &~  \left(1+ \frac{\|g''\|^2_{\L^\infty}}{c_K^2}\right)(\Delta_{\sigma_v}-\Delta_{\sigma_u}) \\
 \le &~ \frac1{c_K}\left(1+ \frac{\|g''\|^2_{\L^\infty}}{c_K^2}\right)(\sigma_v - \sigma_u),
 \end{split}
 \]
 where in the last inequality we used that the Lipschitz constant of $(g')^{-1}$ is bounded by $1\over c_K$.
 Thanks to (\ref{e13}) we now have
 \bel{e14}
  \int_{a^*}^{b'}\bigl|\bfy^{\bar v}(x)-\bfy^{\bar u}(x)\bigr| \,dx ~\le~ \frac1{c_K}\left(1+ \frac{\|g''\|^2_{\L^\infty}}{c_K^2}\right)(b-a)(\sigma_v - \sigma_u)~=~\frac1{c_K}\left(1+ \frac{\|g''\|^2_{\L^\infty}}{c_K^2}\right)\delta.
 \eeq
 \v
 {\bf 7.} {\bf Estimate on $\bigl[\max\{a^*,b'\},\,b^*\bigr]$.} If $\max\{a^*,b'\}<x<b^*$, then $z^{\bar v}(x) \in \,]a,b[\,$ and $\bfy^{\bar v}(x)\ge b$.
In particular $\bfy^{\bar v}(x) - z^{\bar v}(x)= \Delta_{\sigma_v}$ where $g'(-\Delta_{\sigma_v})=\sigma_v$. Therefore
\[
b^* - \max\{a^*,b'\} \,\le\, \Delta_{\sigma_v}
\]
and hence 
$$
\int_{\max\{a^*,b'\} }^{b^*} \frac1{\bfy^{\bar v}(x) - z^{\bar v}(x)} dx~ \le~ 1.
$$
In view of (\ref{eLip}), this yields
\bel{e15} \int_{\max\{a^*,b'\} }^{b^*} \bigl|\bfy^{\bar v}(x)-\bfy^{\bar u}(x)\bigr| \,dx
~\leq~ {2\delta\over c_K}   \int_{\max\{a^*,b'\} }^{b^*} \frac1{\bfy^{\bar v}(x)-z^{\bar v}(x)}dx ~\leq~{2\delta\over c_K} \,.
\eeq
Together, (\ref{e12}), (\ref{e14}) and (\ref{e15}) yield \eqref{Liparr}, completing the proof.
 \endproof

%
%
%
%
%
%
%
%

\section{Departure rates as Nash equilibria}
\label{sec:5}
\setcounter{equation}{0}
We now consider a cost $\vp(x)$ for early departure
and a cost $\psi(x)$ for late arrival (see Fig.~\ref{f:tf273}).  The main results will be obtained under the assumptions
\begin{itemize} 
\item[{\bf (A2)}] {\it  For some constants $L_\vp,L_\psi$, the cost functions $\vp,\psi$ satisfy
\bel{costp1}
-L_\vp\leq \vp'<0,\qquad \vp''\,\geq\,0,\qquad\psi,\psi''~> ~0,\qquad 0<\psi'<L_\psi\,,\eeq
\bel{costp2}
\lim_{x\to +\infty} 
\Big(\vp(x)+ \psi(x)\Big)=+\infty.\eeq}
\end{itemize}

\begin{figure}[htbp]
   \centering
 \includegraphics[width=0.5\textwidth]{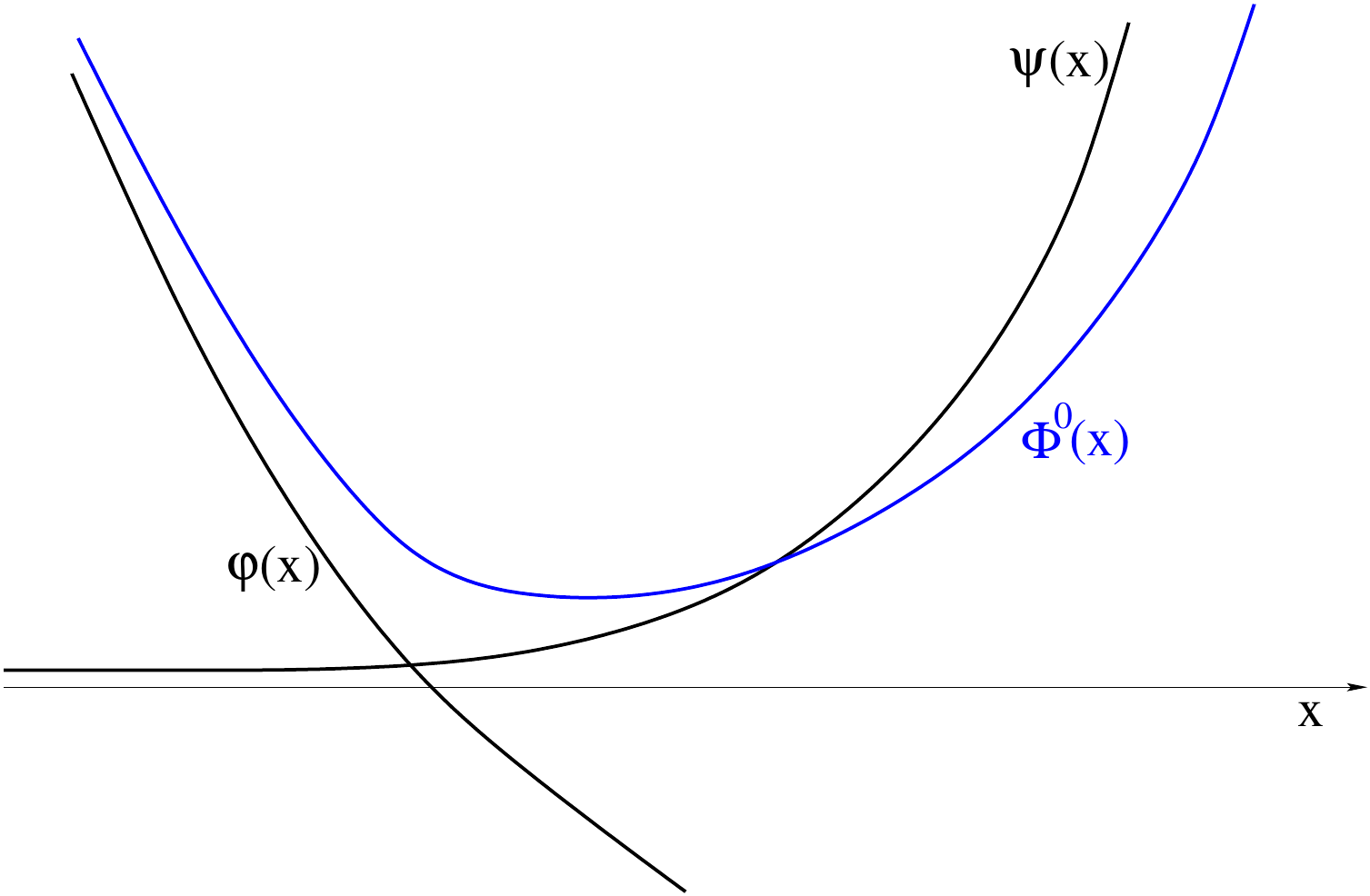}
   \caption{\small The departure cost $\vp(x)$, the arrival cost $\psi(x)$, and a lower bound 
   on the total cost $\Phi^0(x)=\vp(x)+\psi(x+\ell_0)$, according to the assumptions in {\bf (A2)}. }
   \label{f:tf273}
\end{figure}

Given a departure rate $\bar u$, the total cost to a driver departing at time $x$ is defined as
\bel{fud}
\Phi^{\bar u}(x)~\doteq~\vp(x)+ \psi\bigl(y^{\bar u}(x)\bigr),\eeq
where $y^{\bar u}(x)$ is the arrival time defined at (\ref{atime}).

We
denote by $\ell_0, \ell_{max}$ a lower and an upper bound for the duration of  
a trip, as in (\ref{ell0}) and (\ref{ellmax}).  These constants depend only on the flux function $f$ and on the total number $\kappa$ of drivers.
For any departure rate $ \bar u$ with $\|\bar u\|_{\L^1}=\kappa$, the corresponding cost $\Phi^u(x)$ to a driver 
departing at time $x$ thus satisfies the bounds
\bel{Phib} \Phi^0(x)~\doteq ~\vp(x) + \psi(x+\ell_0) ~\leq~\Phi^{\bar u}(x)~\leq ~\vp(x) + \psi(x+\ell_{max}) 
~\doteq~\Phi^{max}(x).\eeq

\begin{remark}\label{r:41} {\rm Choosing $L_\vp$ and $L_\psi$ in (\ref{costp1}) to be the Lipschitz constants of the functions $\vp,\psi$,
by convexity it follows
\bel{limpp}L_\vp\,=\,-\lim_{x\to -\infty} \vp'(x),\qquad\qquad L_\psi\,=\,\lim_{x\to +\infty} \psi'(x).\eeq
From the assumptions {\bf (A2)} it also follows that both functions $\Phi^0, \Phi^{max}$ are convex.
Moreover, there exist the limits of the gradients
\bel{LP1} \lim_{x\to -\infty} \Phi^0_x(x) ~=~ \lim_{x\to -\infty} \Phi^{max}_x(x) ~=~-L_\vp,\eeq
\bel{LP2} \lim_{x\to +\infty} \Phi^0_x(x) ~=~ \lim_{x\to +\infty}  \Phi^{max}_x(x) ~\doteq~L_\Phi~>~0,\eeq
for some positive constant $L_\Phi$.
}
\end{remark}
Let $\ov u(\cdot)$ be a departure rate where all drivers incur   the same total cost $c$.
Calling 
\bel{elldef} \ell(x) ~=~\hbox{[duration of the journey of the driver that departs at time~$x$],}\eeq 
this means
\bel{ccost}
\vp(x)+\psi\bigl(x+\ell(x)\bigr)~=~c\eeq
for all times $x$ at which some driver departs.
Differentiating (\ref{ccost}) one obtains
\bel{dco}
\vp'(x) +\psi'\bigl(x+\ell(x)\bigr)\bigl(1+\ell'(x)\bigr)~=~0,\eeq
\bel{dco2} \vp''(x) +\psi''\bigl(x+\ell(x)\bigr)\bigl(1+\ell'(x)\bigr)^2 + \psi'\bigl(x+\ell(x)\bigr)\ell''(x)~=~0.\eeq
This yields
\bel{lpr}
\ell'(x)~=~{-\vp'(x)\over\psi'\bigl(x+\ell(x)\bigr)} -1,\eeq
\bel{lp2} \ell''(x)~=~-{ \vp''(x) +\psi''\bigl(x+\ell(x)\bigr)\bigl(1+\ell'(x)\bigr)^2 \over  \psi'\bigl(x+\ell(x)\bigr)}~\leq~0.\eeq

\begin{figure}[htbp]
   \centering
 \includegraphics[width=0.5\textwidth]{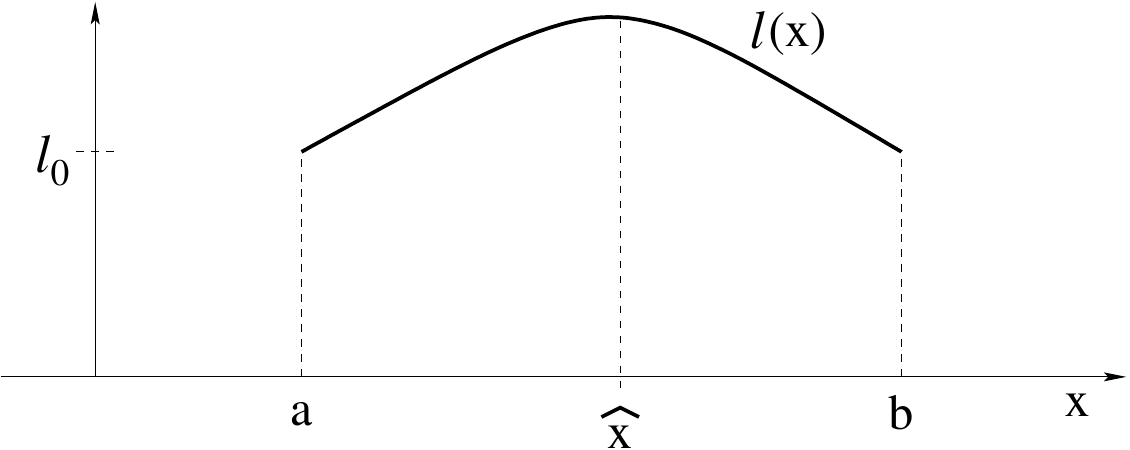}
   \caption{\small   The duration of the trip $\ell(x)$ as a function of the departure time. In a Nash equilibrium this is implicitly defined by (\ref{ccost}).  By (\ref{lp2}) this function is concave down. }
   \label{f:tf257}
\end{figure}
Observing that the map $x\mapsto x+\ell(x)$ is increasing, and so is $\psi'$, we conclude that
\begi
\item All departures occur within a time  interval  $[a,b]$.
\item One has $\ell(a)=\ell(b) = \ell_0 $, the constant in (\ref{ell0}).
\item There exists $\Hat x\in \,]a,b[\,$ such that the duration $\ell(x)$ of the journey is increasing on $[a, \Hat x]$ and decreasing
on $[\Hat x, b]$.  The driver whose journey lasts longer is the one who arrives at the time
$\Hat x +\ell(\Hat x)$, characterized by the relation
$$\vp'(\Hat x) +\psi'\bigl(\Hat x +\ell(\Hat x)\bigr)~=~0.$$
\item At the endpoints, by (\ref{lpr}) we have
$$\ell'(a)~=~{-\vp'(a)\over \psi'(a+ \ell_0 )}-1~>~0,\qquad\qquad \ell'(b)~=~{-\vp'(b)\over \psi'(b+ \ell_0 )}-1~<~0.$$
 \endi

\begin{figure}[htbp]
   \centering
 \includegraphics[width=1.0\textwidth]{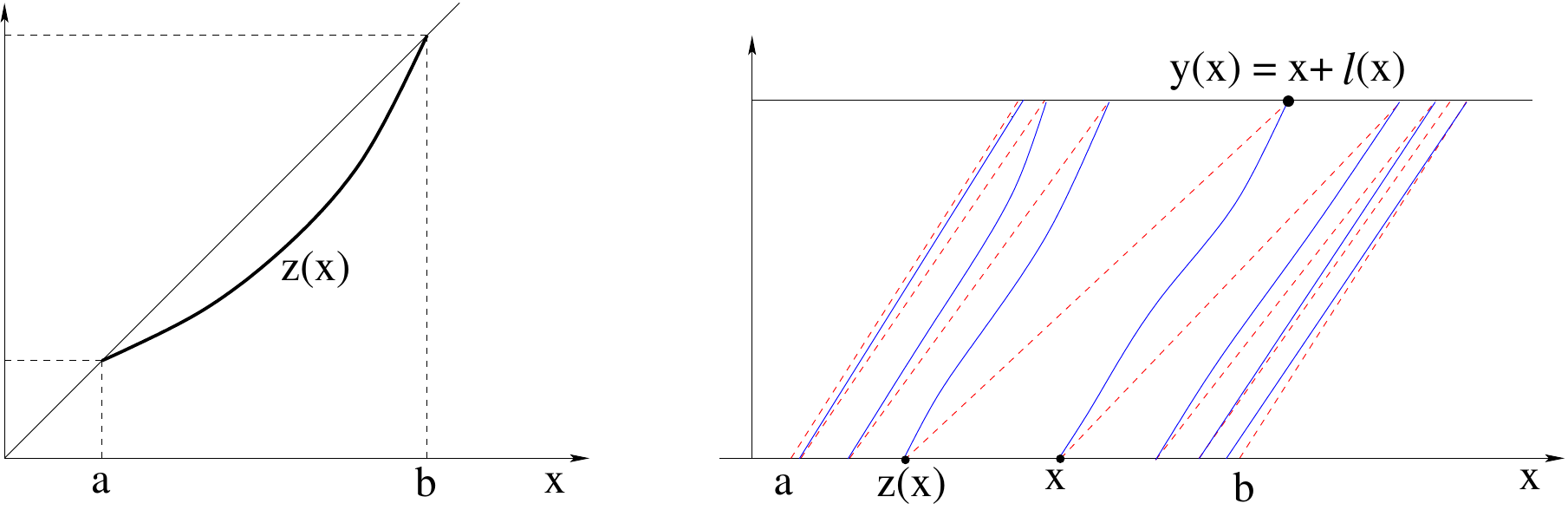}
   \caption{\small  Left: the function $z(x)$  in (\ref{z}). Right: the characteristics and the car trajectories. 
  }
   \label{f:tf256}
\end{figure}

The construction of a Nash equilibrium performed in \cite{BH1} goes as follows.
For a fixed constant $c$, let 
$\Q_c$  be the family of all bounded, non-decreasing departure distributions
\bel{Qdef} Q(x)~=~\int_{-\infty}^x \bar u(s)\, ds\eeq
for some integrable $\bar u:\R\mapsto\R_+$,
such that the corresponding solution
 satisfies
\bel{eless}
\vp(x^d(\beta))+\psi(x^a(\beta))~\leq ~c
\qquad\quad \hbox{for a.e.}~\beta\in [0, ~Q(+\infty)]\,.
\eeq
In other words, $\Q_c$ is the family of all departure distributions where every driver pays a total cost $\leq c$.
Then
 the departure distribution defined as a pointwise maximum:
\bel{Q*def}
Q^*(x)~\doteq~\sup\,\Big\{ Q(x)\,;~~Q\in \Q_c\Big\}\eeq
yields 
a Nash equilibrium.

\begin{remark}\label{r:21}{\rm {\bf (A sufficient condition for a partial solution).} From the previous characterization we deduce the following property.
Assume that $Q\in \Q_c$ as in (\ref{Qdef}) is a departure distribution such that
\bel{echx} x^d(\beta)\leq \Tilde x\qquad\implies\qquad 
\vp(x^d(\beta))+\psi(x^a(\beta))~= ~c,
\eeq
$$\Phi^{\bar u}(x)~\geq~c\qquad\qquad\forall x\leq \Tilde x.$$
That means: every driver departing before time $\Tilde x$ pays exactly $c$ as total cost, and there is no way to achieve a lower cost by departing before time $\Tilde x$.
Then the Nash equilibrium $Q^*$ satisfies
\bel{QHN} Q^*(x)~=~Q(x)\qquad\forall x\leq \Tilde x.\eeq
}\end{remark}

\subsection{An example of a Nash equilibrium.}\label{sec:51}
Assume that the flux function satisfies 
\bel{f0} f(u)~=~u+{u^2\over 2}\qquad\forall u\in [0,1].\eeq
To construct an explicit example of  a smooth Nash equilibrium, we start by
defining a departure rate
\bel{f1}\bar u(x)~=~\left\{\bega{cl} \ve \,\cos x\qquad &\hbox{if} \quad |x|<\pi/2,\\[1mm]
 0\qquad &\hbox{if} \quad |x|\geq \pi/2.\enda\right.\eeq
 Then we construct a departure and an arrival cost such that, with  the departure rate $\bar u(\cdot)$, all 
 drivers pay the same total cost.
 
\begin{figure}[htb]
   \centering
 \includegraphics[width=0.4\textwidth]{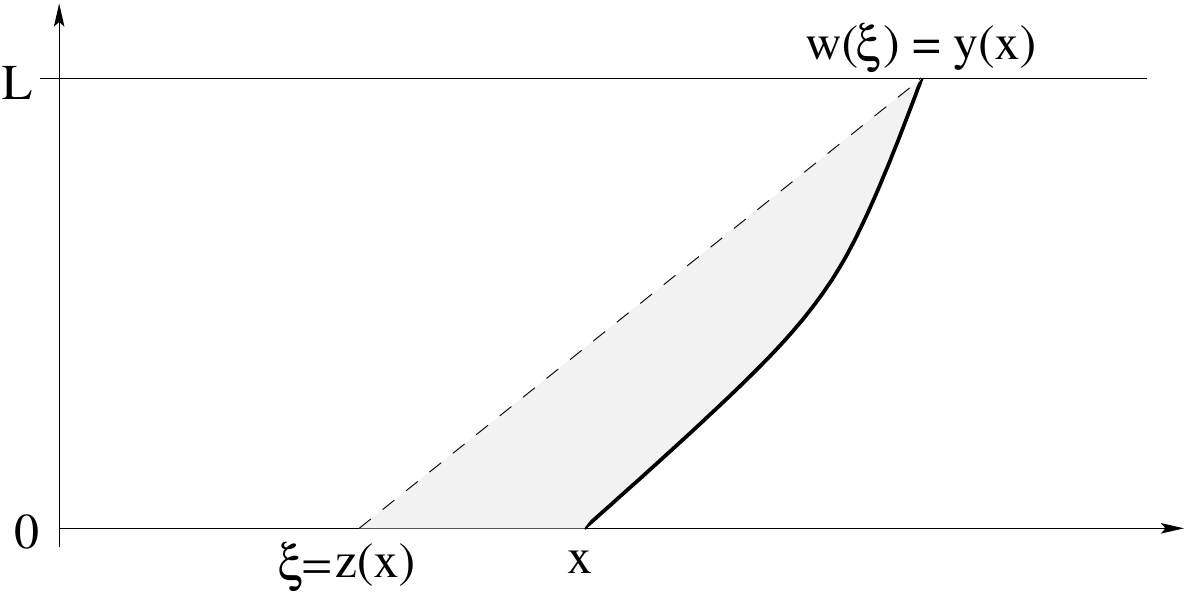}
   \caption{\small Here $y(x)$ is the arrival time of the driver departing at time $x$, while $z(x)$ is the 
   initial point of the backward characteristic through $y(x)$.}
   \label{f:t520}
\end{figure}

To fix ideas, let $L=1$ be the length of the road.   Then the solution to the Cauchy problem 
(\ref{f0})- (\ref{f1}) is obtained by the method of characteristics:
$$u\bigl(1,\, x+1+\bar u(x)\bigr)~=~\bar u(x),$$
provided that
map $x\mapsto 1+x+\ve \cos x$ is smoothly invertible.  This happens for $0<\ve<1$.

The arrival time $y(x)$ of the driver starting at time $x$  is implicitly determined by 
 \bel{eqmass}
\int_{-\infty}^{x} \bar u(s)\,ds \;=\; \int_{-\infty}^{y(x)} u(1,s)\,ds,
\eeq

%
%

Denoting by
$$w(\xi)\,\doteq\,\xi+1+\ve\cos \xi$$
the endpoint of the characteristic starting at $\xi$ (see Fig.~\ref{f:t520}), 
since $u\bigl( 1, w(\xi)\bigr)=\bar u(\xi)$, we have
\bel{inw} \int_{-\infty}^{w(\xi)} u(1,s)\, ds~=~ \int_{-\infty}^\xi \bar u(\zeta)\cdot {dw\over d\zeta}\, d\zeta.
\eeq
Calling $\xi =z(x)$ the initial point of the backward characteristic through $y(x)$,  so that $y(x)= w(\xi)$,
in the particular case where $\bar u$ is the function in (\ref{f1}), from (\ref{eqmass})-(\ref{inw}) we obtain
\[
\int_{-\pi/2}^x\ve\cos s\, ds~=~\int_{-\pi/2}^{z(x)} \ve\cos s\cdot (1-\ve \sin s)\, ds\,,\]
\[
\ve\bigl(\sin x + 1\bigr)~= ~\varepsilon\sin z(x) - \frac{\varepsilon^{2}}{2}\sin^{2} z(x)
  + \varepsilon + \frac{\varepsilon^{2}}{2}\,,
\]
This yields a 
quadratic equation for  $\sin z(x)$ which can then be solved explicitly:
\bel{equad}
\frac{\varepsilon}{2}\sin^2 z(x) - \sin z(x)
+ \Bigl(\sin x - \frac{\varepsilon}{2}\Bigr) ~=~ 0 .
\eeq
The root of this equation which lies inside $[-1,1]$ is computed by
\bel{eqz}
\sin z(x)
\,=\, \frac{1-\sqrt{\,1+\varepsilon^{2}-2\varepsilon\sin x\,}}{\varepsilon}\,,
\qquad
z(x)\,=\,\arcsin\!\left(
\frac{1-\sqrt{\,1+\varepsilon^{2}-2\varepsilon\sin x\,}}{\varepsilon}
\right).
\eeq

In turn, the arrival time of the driver departing at time $x$ is given by
\bel{atx}
y(x)~=~z(x) + 1 + \ve\cos z(x).\eeq
At this stage, any couple of functions $\vp,\psi$ such that
$$\vp(x) + \psi(y(x))~\left\{ \bega{rl} =\,c\quad\hbox{if} ~~|x|\leq \pi/2,\\[1mm]
\geq\,c\quad\hbox{if}~~ |x|> \pi/2,\enda\right.$$
for some constant $c$
would yield a Nash equilibrium. 
For example, we can start by choosing 
$$\psi(y)= y\qquad  \hbox{for all arrival times} \quad y\in \left[1-{\pi\over 2}\,,~1+{\pi\over 2}\right],$$
and then define
$$\vp(x)\doteq c-y^{\ov u}(x)\qquad\qquad |x|\leq \pi/2.$$
We check that this definition yields a convex function $\vp$. Indeed,
when $\psi(x)=x+c$, by (\ref{dco2}) it follows
$$\vp''(x)~=~-\ell''(x)~=~-y''(x).$$
 It remains to check that
\bel{conc}
y''(x)\leq 0\qquad\forall |x|\leq \pi/2.\eeq

Calling $S= \sin z(x)$,  for $|x|< \pi/2$ we find
$$y'(x)~=~{\bar u(x)\over\bar u(z(x))} ~=~{\cos x\over\cos z(x)}~>~0.$$
Differentiating the balance of mass, we obtain
$$z'(x)= {\cos x\over \cos z(1-\ve S)},\qquad  y''(x) ~=~{S\cos^2 x - \sin x \cos^2 z(1-\ve S)\over \cos^3 z(1-\ve S)}$$
Using the identity $\sin x = S+{\ve\over 2} (1-S^2)$ one obtains 
$$y''(x)~=~{-(\ve/4) (1-S^2)^2 (2-\ve S)\over \cos^3 z(1-\ve S)}~<~0.$$
%

\section{Analysis of the Nash equilibrium}    
\label{sec:6}
\setcounter{equation}{0}
Let the function $\ell(x)$ in be  as in Fig.~\ref{f:tf257}, with
$$\ell(a)~=~\ell(b)~=~\ell_0, \qquad\qquad \ell''(x)<0 \quad \hbox{for}~~a<x<b.$$
Note that there exists $a<\Hat x<b$ such that
\bel{lprime}
\left\{\bega{rl} \ell'(x)\geq 0\quad\hbox{if} ~~x\in [a,\Hat x],\\[1mm]
 \ell'(x)\leq 0\quad\hbox{if} ~~x\in [\Hat x,b].\enda\right.\eeq
We need to construct a departure distribution $u(0,x)=\bar u(x)$ supported on the time interval $[a,b]$,
where the driver departing at time $x$ arrives at time $y(x)=x+\ell(x)$. 

As shown in Fig.~\ref{f:tf256}, call $z(x)$ the initial point of the characteristic reaching $y(x)$ at $t=L$.
We have the two basic equations
\bel{beq}\left\{\bega{rl} x + \ell(x)&=~z(x) + L f'\bigl(\bar u(z(x))\bigr),\\[2mm] \bar u(x)&=
~\bigl(1+\ell'(x)\bigr) \bar u\bigl(z(x)\bigr).\enda\right.\eeq
The second equation is obtained from  the identity
\bel{massid}\int_{-\infty}^x u(0,s)\, ds~=~\int_{-\infty}^{y(x)} u(L,s)\, ds.\eeq
Indeed, differentiating w.r.t.~$x$ and observing that the solution $u$ is constant along characteristics, one obtains
\bel{massi2}u(0,x)~=~y'(x) u\bigl(L,y(x)\bigr)~=~\bigl(1+\ell'(x)\bigr) u\bigl(0, z(x)\bigr).\eeq

%
%

If we fix an initial  point $x_0\in \,]a,b[\,$ and a value $u(x_0)=u_0$, the two equations in (\ref{beq})
yield a sequence of points $(x_n, u_n)$,
with
\bel{xzn}
x_{n-1}\,=\,z(x_n),\qquad u_n\,=\,\bar u(x_n),\qquad\qquad n\in {\mathbb Z}.\eeq
By (\ref{beq}), these can be
 inductively constructed by
\bel{seq}\left\{\bega{rl} x_{n-1}+ L f'(u_{n- 1})&\ds=~x_n+ \ell(x_n) \,,\\[4mm]
u_{n-1}&=~ \ds {u_n\over 1+\ell'(x_n)}\,,
\enda\right.\eeq

Ideally, we seek an initial value $u_0$ such that the above sequence satisfies
\bel{limn}\left\{ \bega{rl} x_n\to a,\qquad u_n\to 0\qquad\quad &\hbox{as}~~n\to -\infty,\\[1mm]
 x_n\to b,\qquad u_n\to 0\qquad\quad &\hbox{as}~~n\to +\infty.\enda\right.\eeq
Moreover, the sequence $(x_n)_{n\in{\mathbb Z}}$ should be monotone increasing.

\begin{remark}\label{r:51} {\rm
A word of caution is needed here.  Assume that the time $x_{n-1}$ is given, together with the 
density of departures $u_{n-1}= \bar u(x_{n-1})$.   This uniquely determines the endpoint of the characteristics starting at $x_{n-1}$, namely
\bel{yn3}A_n~=~x_{n-1} + L f'(u_{n-1}).\eeq
As shown in Figure~\ref{f:tn6},  the point $A_n$ is obtained by shifting the graph
of $h(\cdot)$ until it is tangent to the graph of $\ov U$ at the point $\bigl(x_{n-1}, \ov U(x_{n-1})\bigr)$.

Notice  that the two values $(x_{n-1}, u_{n-1}) $ uniquely determine $A_n$, 
but not the point $x_n$ such that $A_n = y(x_n)$.
Indeed, there are many ways to prolong the graph of $\ov U(\cdot)$ for $x>x_{n-1}$,
drawn in red in Fig.~\ref{f:tn6}.   In these various extensions, the driver who arrives at time $A_n$  departs at different times $x_n$.
Motivated by the Nash equilibrium, the choice of $x_n$ is dictated by the identity
\bel{xn4} A_n~=~y(x_n)~=~x_n +\ell(x_n),\eeq
where $\ell$ is the function at (\ref{ccost}).
This yields the first identity in (\ref{seq}).   However, this is meaningful only as long as
\bel{xn5}x_n~\in~[x_{n-1}, A_n -\ell_0].\eeq
As a consequence, the inductive algorithm (\ref{seq}) breaks down as soon as (\ref{xn5}) 
is no longer satisfied.}
\end{remark}

\begin{figure}[htbp]
   \centering
 \includegraphics[width=0.7\textwidth]{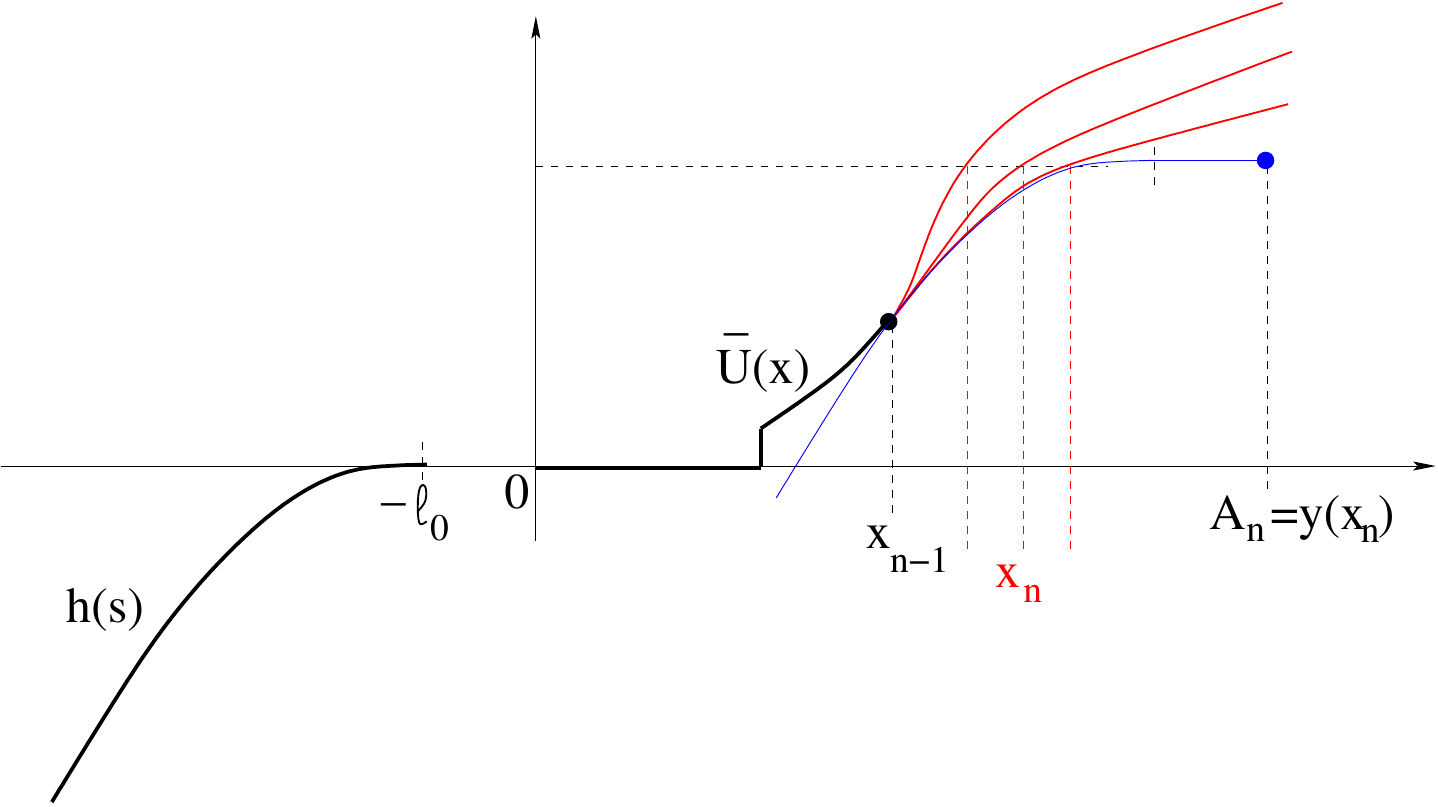}
   \caption{\small  The construction in Remark~\ref{r:51}. }
   \label{f:tn6}
\end{figure}

\v
Recalling (\ref{ell0}), we write the Taylor expansion for $f$ at $u=0$ and for $\ell$ at $x=a$ and at $x=b$:
\bel{Tay}\bega{rl} f(u)&\ds=~{\ell_0\over L} u + {\omega\over 2} u^2+\O(1) \cdot u^3,
\\[4mm]
\ell(x) &\ds=~\ell_0 + c_1 (x-a) + \O(1)\cdot (x-a)^2,\\[4mm]
\ell(x) &\ds=~\ell_0 - c_2 (x-b) + \O(1)\cdot (x-b)^2,\enda\eeq
for some constants $c_1,c_2,\omega>0$, $c_1<1$.

\begin{figure}[htb]
   \centering
 \includegraphics[width=0.6\textwidth]{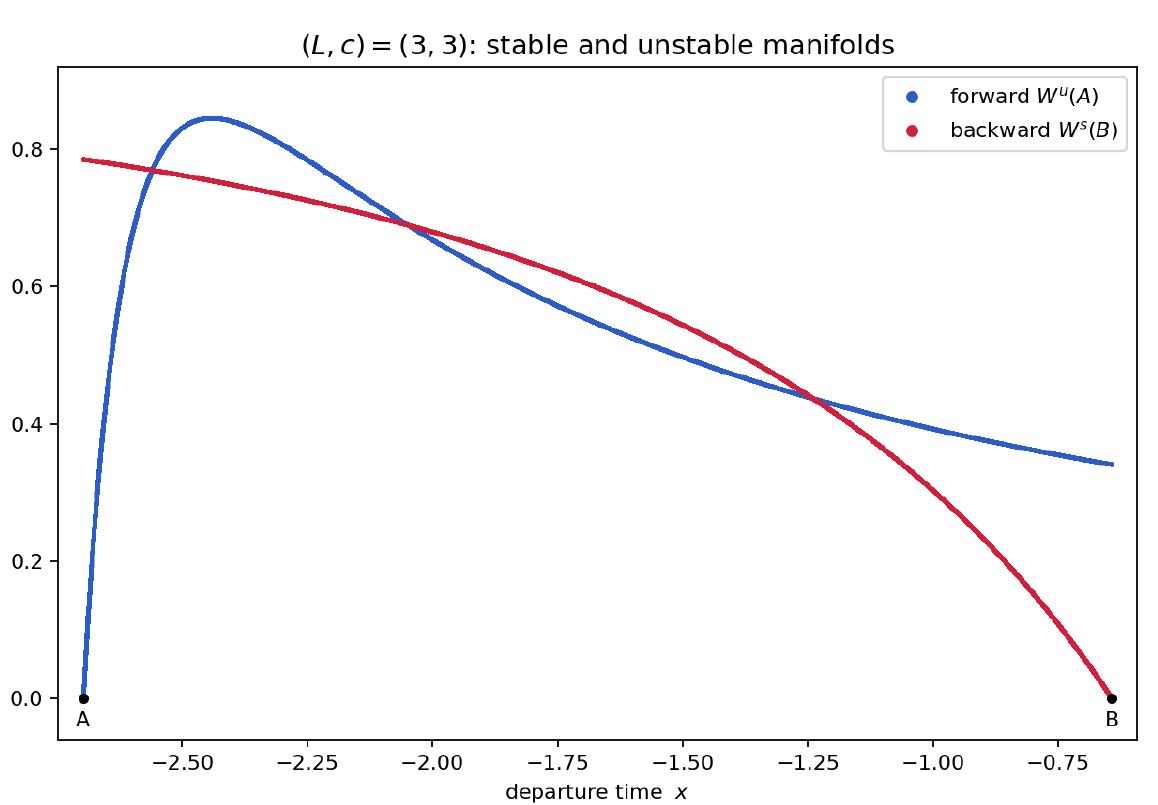}
   \caption{\small  A plot of the unstable manifold through $(a,0)$ and the stable manifold through $B=(b,0)$.
   Here we choose $\vp(t)=-t$, $\psi(t)=e^t$, $f(u)= 1- \sqrt{1-u}$, $L=3$, $c=3$.}
   \label{f:tflow-manifolds}
\end{figure}

Inserting (\ref{Tay}) in (\ref{seq}), for $x\approx a$ we obtain
\bel{Ta2} 
\left\{ \bega{rl} x_{n-1}&\ds \approx x_n + \ell_0 + c_1 (x_n-a) - {L \ell_0\over L} - L \omega\, u_{n-1}\,,
\\[4mm]
u_{n-1}&\ds\approx ~{u_n\over 1 + c_1}\enda\right.\eeq
In matrix notation, this can be written as
\bel{matr1}\begin{pmatrix} x_n-a \cr\cr
u_n\end{pmatrix}~\approx~
\begin{pmatrix}{1\over 1+c_1} && {L\omega\over 1+c_1}\cr\cr
0 && 1+c_1\end{pmatrix}\begin{pmatrix} x_{n-1}-a \cr\cr
u_{n-1}\end{pmatrix}.\eeq
Similarly, for $x\approx b$ we obtain
\bel{Ta4} 
\left\{ \bega{rl} x_{n-1}&\ds \approx x_n + \ell_0 - c_2 (x_n-b) - {L \ell_0\over L} - L \omega\, u_{n-1}\,,
\\[4mm]
u_{n-1}&\ds\approx ~{u_n\over 1 - c_2}\enda\right.\eeq
In matrix notation, this can be written as
\bel{matr2}\begin{pmatrix} x_n-b \cr\cr
u_n\end{pmatrix}~\approx~
\begin{pmatrix}{1\over 1-c_2} && {L\omega\over 1-c_2}\cr\cr
0 && 1-c_2\end{pmatrix}\begin{pmatrix} x_{n-1}-b \cr\cr
u_{n-1}\end{pmatrix}.\eeq
We observe that both matrices in (\ref{matr1}) and (\ref{matr2}) have eigenvalues 
$0<\lambda_1<1<\lambda_2$.  Hence the points $A=(a,0)$ and $B=(b,0)$ are both saddle points for the discrete dynamics.
In a generic setting, we do not expect that the unstable manifold through $A$ 
should coincide with the stable manifold through $B$.
A precise result in this direction will be given in Theorem~\ref{t:82}.
%
%
%
%
%
%

If $\bar u(x)$ is a departure rate yielding a Nash equilibrium, then for every $x\in [a,b]$ we have
\bel{Nchar}
\int_a^x \bar u(y)\, dy~=~\int_{a+ \ell_0 }^{x+\ell(x)} u(L,\zeta)\, d\zeta\,.\eeq
Indeed, if the driver starting at time $x$ arrives at time $y(x)\doteq x+\ell(x)$, that means that the total number of drivers
departing before time $x$ is the same as the total number of drivers arriving before time $x+\ell(x)$.
Assuming that the solution $u=u(t,x)$ remains smooth on the entire domain $[0,L]\times\R$,
the arrival rate $u(L,\cdot)$ can be computed by characteristics.
\v
{\bf 1.}
Let $z(x)$ be the initial point of the characteristic that ends at $y(x)\doteq x+\ell(x)$, so that
\bel{z} z(x)~=~x+\ell(x)- L\,f'\bigl( \bar u(z(x))\bigr) .\eeq
 Applying the divergence theorem to the triangular region with vertices at $(0,z(x))$, $(0,x)$ and $(L, y(x))$ (see Fig.~\ref{f:tf256}, right),
we obtain
\bel{div}
\int_{z(x)}^x \bar u(s)\, ds~=~L \Big[f'\bigl(\bar u(z(x))\bigr) \bar u(z(x))- f\bigl(\bar u(z(x))\bigr) \Big].\eeq
Indeed, the flux across the car trajectory starting at $x$ vanishes identically.
\v
{\bf 2.}
Solving the conservation law by characteristics, we find that the 
gradient satisfies
$${d\over dt} u_x(t, x(t))~=~- f''\bigl(\bar u(x(0))\bigr) u_x^2\bigl(t, x(t)\bigr),\qquad\qquad {d\over dt} x(t)~=~f'(u).$$
\bel{uxdec}u_x(t, x(t))~=~{ \bar u_x(x(0))\over 1+ t f''\bigl(\bar u(x(0))\bigr)  \bar u_x(x(0))}\qquad\qquad t\in [0,L].\eeq
In particular, the
density of arrivals at the time $y(x)$ satisfies
\bel{uxL}u_x(L, y(x))~=~{ \bar u_x(z(x)) \over  1+ L f''\bigl(\bar u(z(x))\bigr) \, \bar u_x(z(x))}\,.\eeq
\v
{\bf 3.}
Differentiating (\ref{z}) w.r.t.~$x$, one obtains
\bel{zp5}
z'(x)~=~1+\ell'(x) - L f''\bigl(\bar u(z(x))\bigr)\, \bar u_x(z(x))\, z'(x),\eeq
\bel{zp6}z'(x)~=~{1+\ell'(x)\over 1+ L f''\bigl(\bar u(z(x))\bigr) \bar u_x(z(x))}\,.\eeq
On the other hand, differentiating (\ref{div}) we find
$$\bar u(x) -\bar u(z(x)) z'(x) ~=~L f''(\bar u(z(x))) \bar u_x(z(x)) \bar u (z(x)) z'(x),$$
\bel{Nase}
\bar u(x)~=~ \Big[1+ L f''\bigl( \bar u(z(x))\bigr) \,\bar u_x(z(x))  \Big] \bar u(z(x))\cdot z'(x)~=~
(1+\ell'(x))\, \bar u(z(x)).    \eeq
\v
{\bf 4.}
We now analyze what happens at the endpoints of the interval $[a,b]$.
At $x=a$ or $x=b$ we have $\bar u=0$. Therefore
$$u_x(L,a + \ell_0)~=~\Tilde \alpha~\doteq~{1\over f''(0) L+{1\over \alpha}},\qquad\qquad u_x(L,b + \ell_0)
~=~\Tilde \beta~\doteq~{1\over f''(0) L+{1\over \beta}}.$$
At $x=a$ we have
\bel{id3}z(a)\,=\,a,\qquad \bar u(a)\,=\,0,\qquad \qquad z'(a)~=~{1+\ell'(a) \over 1+ L f''(0) \bar u_x(a)}\,.\eeq
Indeed, the first identity is clear, because the first driver encounters no traffic, hence travels with characteristic speed: $\ell(a)= \ell_0 $.  The second identity follows by (\ref{Nase}), because $\ell'(x)\not= 0$. The third identity 
follows by (\ref{zp6}).
Differentiating (\ref{Nase}) one obtains
\bel{zprime}\bar u_x(a)~=~(1+\ell'(a)) \bar u_x(a) z'(a),\qquad\qquad z'(a)~=~{1\over 1+\ell'(a)}~<~1,
\eeq
\bel{bux} 1+ L f''(0)\bar u_x(a) ~=~{1+\ell'(a)\over z'(a)},\qquad\qquad \bar u_x(a)~=~{(1+\ell'(a))^2-1\over L f''(0)}~>~0\,.\eeq
Similarly, at $x=b$ we have
$$\bar u(b)\,=\,0,\qquad z(b)\,=\,b,\qquad\qquad z'(b)~=~{1+\ell'(b) \over 1+ L f''(0) \bar u_x(b)}\,.$$
Differentiating (\ref{Nase}) one obtains
\bel{zpb}z'(b)~=~{1\over 1+\ell'(b)}~>~1,\qquad\qquad 
 \bar u_x(b)~=~{(1+\ell'(b))^2-1\over L f''(0)}~<~0\,.\eeq

\begin{figure}[htb]
   \centering
 \includegraphics[width=1.0\textwidth]{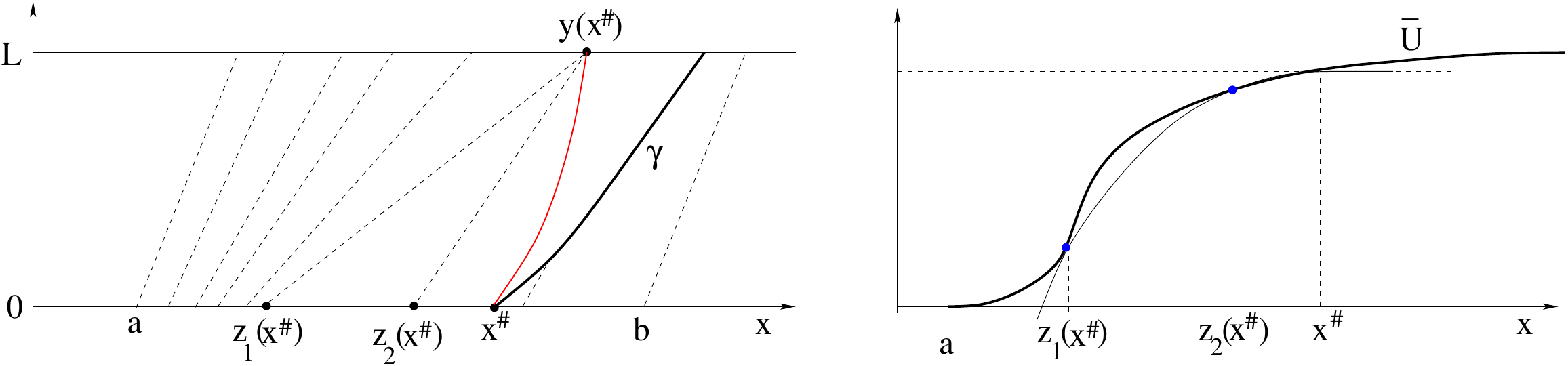}
   \caption{\small  Left: if $x^\sharp$ is a point where a jump in the Nash departure rate 
   occurs, then two 
   distinct characteristics should meet at the arrival time
   $y(x^\sharp)$.   Right: the 
   initial points $z_1<z_2$ of these two characteristics correspond to the points where
   the function at (\ref{hlower}) is tangent to the graph of the integrated function $\ov U$.}
   \label{f:tn4}
\end{figure}

%
%
%

Calling $x^\sharp\in \,]\Hat x, b[\,$ the position of this first shock, as shown in Fig.~\ref{f:tn4}, left,
there should be two distinct
characteristics reaching $y^{\bar u}(x^\sharp)$, starting from two points
\bel{zdag12}z^{\bar u}_1(x^\sharp)~<~z^{\bar u}_2(x^\sharp).\eeq
This is the case if
\bel{cond1} z^{\bar u}_1(x^\sharp)+ L f' \Big(\bar u \big(z^{\bar u}_1(x^\sharp)\bigr)\Big)
~=~ z^{\bar u}_2(x^\sharp)+ L f' \Big(\bar u \big(z^{\bar u}_2(x^\sharp)\bigr)\Big).\eeq
In addition, applying the divergence theorem on the triangle with vertices
$(0, z_1)$, $(0, z_2)$ and $(L,y)$, one obtains the identity
\bel{cond2}
\int_{z_1}^{z_2} \bar u(x)\, dx~=~L\Big[ \bar u(z_1) f' \bigl(\bar u(z_1)\bigr)- f\bigl(\bar u(z_1)\bigr)  \Big] 
- L\Big[\bar u(z_2) f' \bigl(\bar u(z_2)\bigr)- f\bigl(\bar u(z_2)\bigr) \Big].
\eeq

In terms of the integrated function $\ov U$ in (\ref{Udef}), this means that 
the function 
\bel{hlower} s~\mapsto ~\ov U(x^\sharp) +h(s-y^{\bar u} x^\sharp)\bigr)\eeq
is tangent from below to the graph of $\ov U$ at the two distinct points (\ref{zdag12}).

\section{Emergence of a queue}
\label{sec:7}
\setcounter{equation}{0}
Let the departure and arrival cost functions $\vp, \psi$ be given, as in {\bf (A2)}.

Let  $\rho$ be the density of cars, and let $\rho v(\rho)$ be the flux function. Set
$$v_0 = v(0),\qquad v^* = v(\rho^*),$$
where $$\rho^*\,=\,\argmax_\rho \rho v(\rho)$$
is the density which yields the maximum possible flux.  

Moreover, set  $L$ = length of the road,~  $\kappa=$ total number of drivers.

\begin{proposition}\label{p:71} There exist $\kappa^+>\kappa^->0$ such that:
\begi
\item[(i)] If $\kappa>\kappa^+$ the Nash equilibrium solution contains a queue and a shock.
\item[(ii)] If $\kappa<\kappa^-$ the Nash equilibrium solution does not contain any queue. 
\endi
\end{proposition}
\v
{\bf Proof.} 
{\bf 1.} In this first step we claim that, if the total amount of cars $\kappa$ is sufficiently large,
then a queue must form at the entrance of the road.

Indeed, observe that
 the car speed is bounded above and below:
$$v(t,x)~\in~[v^*, v_0].$$
Hence the travel time  by any driver satisfies
$$T~\in~\left[ {L\over v_0}\,,~{L\over v^*}\right].$$

By the assumptions {\bf(A2)}, there exists a departure time $x_0$ such that
\bel{x0def}x_0~=~\hbox{arg}\!\min_{x\in \R} \left\{ \vp(x) + \psi\Big(x+{L\over v^*} \Big)\right\}.\eeq

Any driver can choose to depart at time $x_0$.  If no queue is present, this driver will immediately enter the 
highway, without having to wait.
In this case, he will arrive at time $x^{arrival} \leq x_0 +{L\over v^*} $.   Hence his total cost will
be 
$$\Phi(x_0)~\leq~C_0\doteq ~ \vp(x_0) + \psi\Big(x_0+{L\over v^*} \Big).$$

On the other hand,  still by {\bf (A2)},  we can choose a constant $K$ large enough so that 
$$\vp(x)+ \psi\Big(x+{L\over v_0} \Big)~>C_0\qquad\qquad\forall x\notin [-K, K].$$
That means: every driver who departs at a time $x\notin [-K, K]$ must pay a total cost $> C_0$.   
This implies that in a Nash equilibrium all drivers must depart at a time $x\in [ -K,K]$.

But if the total number of drivers $\kappa$ satisfies
\bel{kk+}\kappa~>~\kappa^+~\doteq~2K \rho^*v(\rho^*),\eeq
this is  not possible.    Since the maximum flux at the entrance of the highway is $M=\rho^* v(\rho^*)$, the right hand side is the maximum number of drivers that can depart
during a time interval of length $2K$.   We thus reached a contradiction.
\v
{\bf 2.} By the previous step, if (\ref{kk+}) holds the flux $\bar u(x)$ entering the highway reaches the maximum value $M=
\rho^* v(\rho^*)$ at some time $\bar x$ (see Fig.~\ref{f:tf260}).   
Consider the time
$$x^\sharp~\doteq~\sup~\bigl\{ x\geq \bar x\,;~\bar u(x)= M\bigr\}.$$
Since the characteristic speed satisfies $f'(M)=+\infty$, as shown in Fig.~\ref{f:tf261} this 
implies that the solution contains a shock, originating at the point $x^\sharp$.
This proves part (i) of the Proposition.
\v

\begin{figure}[htb]
   \centering
 \includegraphics[width=0.6\textwidth]{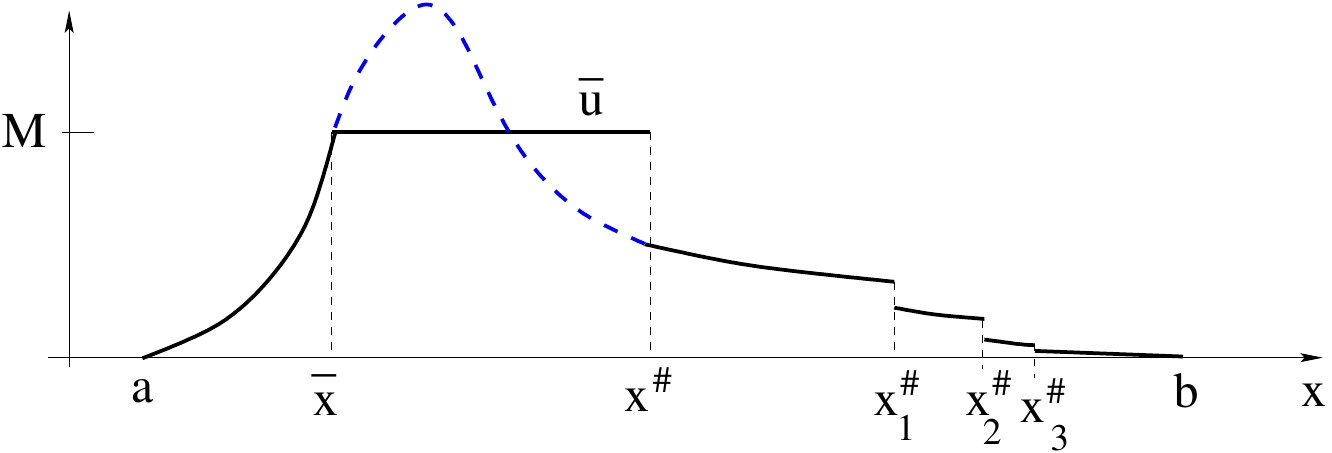}
   \caption{\small  A departure rate in a Nash equilibrium. The dashed curve is the rate at which drivers start their
journey, possibly joining the queue at the entrance of the highway.  The solid curve is the flux of cars entering the highway.  The queue is formed at time $\bar x$ and depleted at time $x^\sharp$.    The departure rate has 
downward jumps at points $x_1^\sharp<x_2^\sharp<\cdots$, determined as in Fig.~\ref{f:tf261}.}
   \label{f:tf260}
\end{figure}

\begin{figure}[htb]
   \centering
 \includegraphics[width=0.6\textwidth]{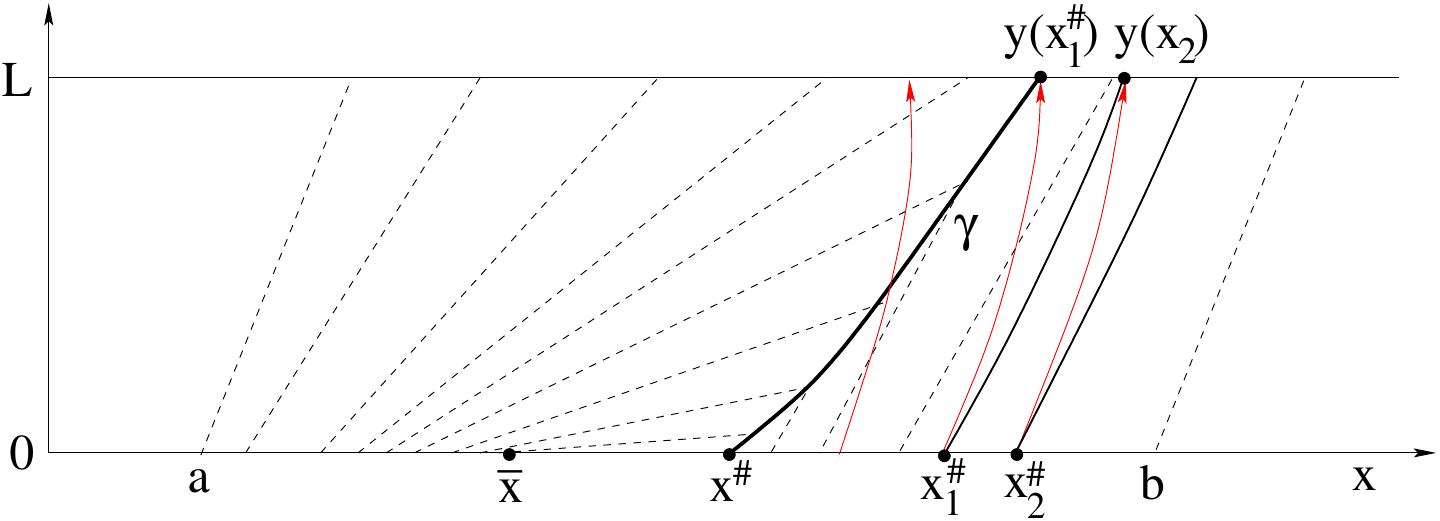}
   \caption{\small  The characteristic lines in the solution with initial data $\bar u$= flux entering the highway at time $x$, as in Figure~\ref{f:tf260}.  Here a queue is formed at time $\ov x$ and depleted at time $x^\sharp$.
   At time $ x^\sharp$ a shock is formed, moving along the curve $\gamma$. Here $x_1^\sharp$ is the 
   departure time of the driver whose arrival time $y(x_1^\sharp)$ coincides with the terminal point along the shock.
 By the necessary conditions (\ref{Nase}), that the departure rate has a downward jump at a point $x_1^\sharp$.
 Hence a second shock initiates at this point. By induction, if $\gamma_j$ is the shock curve starting at $x_j^\sharp$, let $x_{j+1}^\sharp$ be the   departure time of the driver whose arrival time $y(x_{j+1}^\sharp)$ coincides with the terminal point along $\gamma_j$.  By (\ref{Nase}) the departure rate has a 
 further downward jump at a point $x_{j+1}^\sharp$. }
   \label{f:tf261}
\end{figure}

{\bf 3.} To prove (ii) we observe that the
formula (\ref{arrtime}) for the arrival time  implies the sharper uniform bound on the travel time:
\[
\ell_0~\le~ y^{\bar u}(x)-x~\le ~h^*(\kappa).
\]
Since $h^*(\kappa)\to\ell_0$ as $\kappa\downarrow 0$, we have
\[
0~\le~\Phi^{\bar u}(x)-\Phi^0(x)
~\le ~L_\psi\bigl(h^*(\kappa)-\ell_0\bigr)~\to~0.
\]
The common cost to all drivers in the Nash equilibrium is the global minimum of $\Phi^{\bar u}$, so
\[
\min\Phi^0\le c(\kappa)
\le\min\Phi^0+L_\psi\bigl(h^*(\kappa)-\ell_0\bigr).
\]
By the strict convexity and the coercivity, the function  $\Phi^0$ has a unique minimizer $x_0$. 
As $\kappa\downarrow 0$, the support of $\bar u$ (i.e., the set of  times when some driver actually departs) must be contained 
in a small neighborhood of $x_0$. All the  arrival times lie near $x_0+\ell_0$.

The Oleinik estimate yields $u(L, s_2)-u(L,s_1)\le(s_2-s_1)/(\omega_0 L)$. If $u(L,s)=r>0$, then
\[
u(L, s-h)~\ge ~r-\frac{h}{\omega_0 L},\qquad
0\le h\le r\omega_0 L.
\]
Integrating both sides of this inequality, since $\int u(L,s)\, ds=\kappa$ we obtain
\[
\bigl\|u(L,\cdot)\bigr\|_{\L^\infty}~\le~\sqrt{\frac{2\kappa}{\omega_0 L}}\,.
\]
On the support of the departure rate $\bar u$, the derivative $y'(x)= -\varphi'(x)/\psi'(y(x))$ is bounded by a constant $R_0$, 
uniformly valid for all  $\kappa>0$ sufficiently small. Recalling (\ref{massi2}) one obtains
\[
\bar u(x)~=~y'(x)u(L,y(x))~
\le ~R_0\sqrt{\frac{2\kappa}{\omega_0L}}
\quad\text{a.e.}
\]
By choosing
\[
\kappa\,<\,\frac{M^2\omega_0L}{2R_0^2}
\]
the above inequality yields $\bar u(x)<M$
showing that the departure rate must be $< M$ almost everywhere. Hence no queue forms at the entrance of the highway.
\endproof

\begin{remark} {\rm The necessary conditions (\ref{beq}) and (\ref{massi2}) imply that if the terminal profile $u(L,\cdot)$
has a jump at a point $y(x^\sharp)$, then the initial departure rate $\bar u$ must have a jump at $x^\sharp$.
In particular, this happens if $y(x^\sharp)$ is the terminal point of the curve $\gamma$ 
along which the solution has a shock, i.e., a downward jump. See Figures~\ref{f:tf260} and \ref{f:tf261}.}
\end{remark}

We observe that Part (ii) does not rule out the presence of a shock in the solution.

\section{Generic failure of regularity of the Nash equilibrium}
\label{sec:8}
\setcounter{equation}{0}

Let $\bar u$ be the departure rate in a Nash equilibrium where each driver pays the same total cost $c$.   In this setting,  the total number of drivers depends on $c$.  
 It is strictly positive as soon as 
 \bel{cmin} c~>~\min_{x \in \R}\varphi(x) + \psi (x+\ell_0).\eeq
As before, for every $x \in [a,b]$ we denote by $y(x)=x+\ell(x)$ the arrival time of the driver starting at time $x$. 
According to (\ref{ccost}), this function is implicitly defined by the identity
\bel{yxdef}
\vp(x) + \psi\bigl(y(x)\bigr)\,=\,c\qquad \forall x\in [a,b].\eeq
Notice that under the assumptions {\bf (A2)}, we have that $y(x)= \psi^{-1}(c-\varphi(x))$ is a $\C^2$ concave function, with $\C^2$ inverse.

In this section we prove two main results: the first states the one-sided Lipschitz regularity of
$\bar u$,  the second states that Nash equilibria are generically not continuous.

\begin{theorem}\label{t:81}
Let $f, \varphi, \psi$ be as in {\bf (A1)-(A2)} and, for a given $c$ as in (\ref{cmin}), let $\bar u$ be the Nash equilibrium departure rate where every driver has the same total cost $c$.
Then $\bar u$ is bounded with  compact support. Moreover, the following one-sided
Lipschitz estimate holds: 
\bel{uLip}
\bar u(x_2)-\bar u (x_1)~ \le~ \left(M\|y''\|_{\L^\infty} + {1\over \omega_0 L} \|y'\|^2_{\L^\infty}\right)(x_2-x_1) \qquad \text{for a.e. } x_1<x_2.
\eeq
where $\omega_0$ is the constant in (\ref{fprop}).
\end{theorem}

{\bf Proof.}
For any $\bar x\in [a,b]$ and $\delta>0$, one has 
\bel{e:estQ}
\ov U(\bar x + \delta) - \ov U(\bar x) ~= ~\int_{y(\bar x)}^{y(\bar x + \delta)}u(L,x) dx.
\eeq
Since $u(L,x)\leq M$ for all $x$, this implies that 
$\ov U$ is Lipschitz and 
$$0~\leq~\ov U_x(x) ~\leq~M\, \|y'\|_{\L^\infty}\qquad\forall\quad x \in [a,b].$$

Since $u$ solves the conservation law (\ref{claw1}) with strictly convex flux, by Oleinik's one-sided Lipschitz estimate, for every $ x \in \R$ it holds
\[
u\bigl(L,y( x +\delta)\bigr)-u\bigl(L,y( x)\bigr) ~\le~ { y( x +\delta)-y( x)\over \omega_0 L}~ \le~ \frac{\|y'\|_{\L^\infty}}{\omega_0 L} \delta.
\]
Differentiating \eqref{e:estQ} w.r.t.~$\bar x$, we obtain  
\[
\bega{rl}
\ds \left( \frac{\ov U(\bar x + \delta)-\ov U(\bar x)}\delta\right)' & \le ~\ds\frac1\delta\Big(u\bigl(L,y(\bar x + \delta)\bigr)y'(\bar x + \delta) - u\bigl(L,y(\bar x)\bigr)y'(\bar x)\Big) \\[4mm]
&\ds =~ \frac1\delta\Big(u(L,y(\bar x + \delta))y'(\bar x + \delta) - u(L,y(\bar x + \delta))y'(\bar x)\Big)\\[4mm]
&\qquad \ds+   {1\over\delta}\Big(u(L,y(\bar x + \delta))y'(\bar x)-   u(L,y(\bar x))y'(\bar x) \Big)\\[4mm]
  &\ds \le ~ M \|y''\|_{\L^\infty} + \|y'\|_{\L^\infty}^2{1\over \omega_0 L}\,.
\enda
\]
An integration  over the interval $\bar x\in [x_1, x_2]$ yields
$$
 \frac{\ov U( x_2 + \delta)-\ov U(x_2)}\delta-
 \frac{\ov U(x_1 + \delta)-\ov U(x_1)}\delta~\leq~ \left( M \|y''\|_{\L^\infty} +{ \|y'\|_{\L^\infty}^2\over \omega_0 L}\right)(x_2-x_1).$$
 Taking the limit as $\delta\to 0$ we obtain (\ref{uLip}).
\endproof

The previous result implies that $\bar u$ has bounded variation and at most countably many jump discontinuities.
It also implies that for small total cost $c$ there is no queue. Indeed if 
$$(b-a) \left(M\|y''\|_{\L^\infty} + {\|y'\|^2_{\L^\infty} \over \omega_0 L}\right) ~<~M,$$ then $\bar u(x)< M$ for every $x \in \,]a,b[$.
The next result states that this regularity is generically optimal.

\begin{theorem}\label{t:82}
For every $f$ as in {\bf (A1)} and $c\in\R$, the set $\mathcal J(c)$ of couples $(\varphi, \psi)$ for which the rate of departures $\bar u$ of the Nash equilibrium has a jump type discontinuity is an open dense set in 
\[
\left\{ (\varphi, \psi) \in \C^2(\R)\,;~~ \hbox{{\bf (A2)} holds and}~ \min_{x\in \R}\ \varphi(x) + \psi(x+\ell_0) < c  \right\},
\]
w.r.t.~the $\C^2$ topology.
\end{theorem}

In order to prove the result above we provide necessary conditions on the cost functions $\varphi, \psi$ for having continuous Nash equilibria $\bar u$ and then prove that generically are not satisfied.
We observe that it suffices to consider the case where $\bar u(x)<M$ for all $x$, 
since otherwise $\bar u$ has a jump type discontinuity.

To state the necessary conditions for the continuity of $\bar u$, we consider the map $g: \Omega\subset \R^2 \to \R^2$ defined by
\bel{e:defg}
g(x,v)~\doteq~ \left(y(x) - Lf'\Big(\frac{v}{y'(x)}\Big),~ \frac{v}{y'(x)}\right), 
\eeq
where $y$ is extended to a $\C^2$ concave function with $\C^2$ inverse on the whole $\R$ and 
\bel{Omdef}\Omega~\doteq ~\left\{ (x,v): \frac{v}{y'(x)}<M \right\}.\eeq
The Jacobian matrix of $g$ is 
\bel{Dg}
Dg(x,v)~=~
\begin{pmatrix}
y'(x) +L f''(\frac{v}{y'(x)})\frac{y''(x)v}{(y'(x))^2}&& -L f''(\frac{v}{y'(x)})\frac1{y'(x)}\\[3mm]
-\frac{vy''(x)}{(y'(x))^2} && \frac1{y'(x)}
\end{pmatrix}
\eeq

\begin{lemma}
The map $g$ is a $\C^1$ diffeomorphism on its image and has  two hyperbolic fixed points $(a,0)$ and $(b,0)$ with $a,b$ such that $\ell(a)=\ell(b)=\ell_0$. 
The Jacobian matrix $Dg(a,0)$ has eigenvalues  $\frac1{y'(a)}, y'(a)$, while the matrix $Dg(b,0)$  has eigenvalues  $\frac1{y'(b)}, y'(b)$.
Since $0<y'(b)=\frac{-\varphi'(b)}{\psi'(b+\ell_0)}<1<y'(a)=\frac{-\varphi'(a)}{\psi'(a+\ell_0)}$, the two points $(a,0), (b,0)$ are hyperbolic.
\end{lemma}
{\bf Proof.}
The Jacobian matrix of $g$ has determinant
equal to $1$ for every $(x,v)\in \Omega$. 

To prove that the map $g= (g_1,g_2)$ is one-to-one, we observe that $g(x_1,v_1)=g(x_2,v_2)$ implies $\frac{v_1}{y'(x_1)}=\frac{v_2}{y'(x_2)}$, hence $y(x_1)=y(x_2)$.
Since $y$ is invertible this implies $x_1=x_2$ and hence $v_1=v_2$.
%
The computation of eigenvalues is straightforward.
\endproof

Given $\delta>0$ sufficiently small, we denote respectively by $W^s_\delta(a,0)$ and $W^u_\delta(a,0)$, 
 the intersections of the local stable and unstable manifolds through the point $(a,0)$,
 with the ball  $B_\delta((a,0))$ centered at $(a,0)$ with radius $\delta$. 
Similarly, 
$W^s_\delta(b,0)$, $W^u_\delta(b,0)$  will denote the local stable and unstable manifolds through $(b,0)$, respectively.

\begin{remark}{\rm
It is easy to check that $W^u_\delta(a,0) = \bigl\{ (x,0) ;\,  a-\delta<x<a+ \delta\bigr\}$ and $W^s_\delta(b,0) = \bigl\{ (x,0) ;\, b-\delta<x<b+ \delta\bigr\}$. 
For our purpose, the relevant  manifolds are $W^s_\delta(a,0)$ and $W^u_\delta(b,0)$.}
\end{remark}

\begin{lemma}\label{l:contNash}
Assume that the departure rate $\bar u(\cdot)$ corresponding to a Nash equilibrium 
is continuous.
Then $\bar u$ is Lipschitz continuous with Lipschitz constant
\bel{e:Lipu}
\mathrm{Lip}(\bar u) ~\le ~M\|y''\|_{\L^\infty} + {1\over \omega_0 L} \|y'\|^2_{\L^\infty} \,.
\eeq
Moreover there is a continuous increasing function $z:[a,b]\mapsto [a,b]$ such that
\begin{equation}\label{e:constraints-smooth}
\left\{\bega{rl}
y(x)&= ~z(x) + Lf'\bigl(\bar u(z(x))\bigr), \\[2mm]
\bar u (x) &= ~y'(x)\bar u (z(x)),
\enda\right.\qquad\qquad\forall x\in [a,b].
\end{equation}
In particular for every $x \in \,]a,b[\,$ one has $g\bigl(x,\bar u (x)\bigr)=
\bigl(z(x), \bar u (z(x))\bigr)$ and 
\bel{e:stable-unstable}
\begin{split}
W^s_+(a,0) \doteq &~  \bigcup_{ n \in \N} g^{-n}\bigl(W^s_\delta(a,0)\cap \{v>0\}\bigr)~ =~ \bigl\{(x,\bar u(x));\, x \in \,]a,b[\bigr\},\\
W^u_+(b,0) \doteq &~  \bigcup_{ n \in \N} g^{n}\bigl(W^u_\delta(b,0)\cap \{v>0\}\bigr)
~ = ~\big\{(x,\bar u(x)): x \in \,]a,b[\big\}.
\end{split}
\eeq
\end{lemma}
{\bf Proof.}
{\bf 1.} Differentiating the identity
\[
\int_{-\infty}^x\bar u(s)ds ~=~ \int_{-\infty}^{y(x)}u(L,s)ds
\]
we obtain 
\bel{e:baruuL}
\bar u(x) ~= ~y'(x) u\bigl(L,y(x)\bigr).
\eeq Since the map $x\mapsto y(x)$ is bi-Lipschitz on $[a,b]$ and $y'(x)>0$, this implies that $u(L,\cdot)$ is continuous. 
We already proved that $\bar u$ satisfies \eqref{uLip}.
It remains to prove that $\bar u$ also satisfies a downward Lipschitz estimate.
Since $u$ is a continuous solution of  \eqref{claw1} on $\,]0,L[\, \times \R$, forward characteristics do not cross and hence
\[
f'\bigl(\bar u(x_2)\bigr)-f'\bigl(\bar u(x_1)\bigr)~ \ge\, - \frac1L(x_2-x_1) \qquad \text{for a.e. }~~x_1<x_2\,.
\]
In turn, this implies
\[
\bar u(x_1)-\bar u(x_2)~\leq~ {x_2-x_1\over \omega_0 L} \qquad \text{for a.e. }~x_1<x_2\,,
\]
where $\omega_0$ is the constant in (\ref{fprop}).
\v
{\bf 2.} 
Since $u$ is constant along backward characteristics we also deduce that $\bar u \bigl(z(x)\bigr)= u\bigl(L,y(x)\bigr)$ with $z(x)= y(x)- Lf'\bigl(u(L,y(x))\bigr)$. 
Combined with \eqref{e:baruuL}, this proves \eqref{e:constraints-smooth}.
\v
{\bf 3.}
By the definition of $g$ at (\ref{e:defg}),   for every $x\in \,]a,b[\,$ 
one has $g\bigl(x,\bar u (x)\bigr)=\bigl(z(x),\bar u(z(x))\bigr)$.
In order to deduce \eqref{e:stable-unstable} it is sufficient to observe $\bar u (a)=\bar u(b)=0$ and that $z:[a,b]\mapsto [a,b]$ is invertible and satisfies $z(x)<x$ for every $x \in \,]a,b[\,$.
\endproof

We are now ready for

{\bf Proof of Theorem \ref{t:82}.}

{\bf 1.} 
In order to prove that $\mathcal J(c)$ is open it suffices to observe that $\bar u$ depends continuously on the cost functions $\varphi,\psi$. Indeed if $\bar u_n$ are continuous Nash equilibria for the costs functions $\varphi_n,\psi_n$ and $\varphi_n \to \varphi$ and $\psi_n\to \psi$ in $\C^2$, then the  $u_n$ are uniformly Lipschitz thanks to \eqref{e:Lipu} and therefore $\bar u = \lim_{n\to \infty}\bar u_n$ is Lipschitz as well. 
\v
{\bf 2.} Next, let 
$\varphi, \psi$ be departure and arrival costs as in {\bf (A2)},  such that the corresponding Nash equilibrium $\bar u$ is continuous. We will show that, for any $\ve>0$, there exists a perturbed arrival cost  $\psi_\ve$, with $\|\psi_\ve-\psi\|_{\C^2}<\ve$, for which the corresponding Nash equilibrium $\bar u_\ve$ is not continuous. 

By Lemma~\ref{l:contNash}, $\bar u$ is Lipschitz continuous.
Let $\bar x$ be a point where $\bar u$ is differentiable and let $\delta_\psi>0$ be small enough so that
\[
 y\bigl(z(\bar x)\bigr)<y(\bar x) - \delta_\psi<y(\bar x)+\delta_\psi <y\bigl(z^{-1}(\bar x)\bigr).
\]  
For $\ve>0$, set
\[
\psi_\ve = \psi + \ve \tilde \psi,
\]
for some smooth perturbation with compact support
$$ \tilde \psi \in \C^\infty(y(\bar x)-\delta_\psi, y(\bar x) + \delta_\psi),$$ to be chosen later.

Assume by contradiction that, for every $\ve>0$, the corresponding Nash equilibrium $\bar u_\ve$ is continuous. Then we would have
\bel{assu}
\bar u_\ve(x) = \bar u(x) \qquad \text{for every }~~x \in \bigl[a,z(\bar x)\bigr] \cup [\bar x, b]
\eeq
because the corresponding stable and unstable manifolds do not depend on $\ve$ if $x\in  [a,z(\bar x)] \cup [\bar x, b]$, namely
\[
\begin{split}
W^{s,\ve}_+(a,0)\cap \bigl\{ (x,v) \in (a, z(\bar x)]\times \R \bigr\} = &~ W^s_+(a,0)\cap 
\bigl\{ (x,v) \in (a, z(\bar x)]\times \R \bigr\}, \\
W^{u,\ve}_+(b,0)\cap \bigl\{ (x,v) \in [\bar x,b)\times \R \bigr\} = &~ W^u_+(b,0)\cap 
\bigl\{ (x,v) \in [\bar x,b)\times \R \bigr\}.
\end{split}
\]
\v

{\bf 3.} Consider the map $x\mapsto y_\ve(x)$ implicitly defined by
$$\vp(x) +\psi_\ve\bigl(y_\ve(x)\bigr)~=~c,$$
where $c$ is the common cost to all drivers.  Equivalently,
$y_\ve(x)=\psi_\ve^{-1}(c-\varphi(x))$.    A Taylor expansion yields
\[
\begin{split}
y_\ve(x)= &~ y(x) - \ve \frac{\tilde \psi(y(x))}{\psi'(y(x))}+ \O(\ve^2) \\
\doteq &~ y(x) + \ve \Delta y(x) + \O(\ve^2) \\
y_\ve'(x)= &~ y'(x) +\ve \left[ -\frac{\tilde \psi'(y(x))y'(x)}{\psi'(y(x))} + \frac{\psi''(y(x))y'(x) \tilde \psi(y(x))}{(\psi'(y(x)))^2}\right] + \O(\ve^2)\\
\doteq &~ y'(x)+ \ve \Delta y'(x) + \O(\ve^2).
\end{split}
\]
Therefore, for any $v$ as in (\ref{Omdef}), we have
\[
\frac{v}{y'_\ve(x)}= \frac{v}{y'(x)}-\ve \frac{v}{y'(x)^2}\Delta y'(x) + \O(\ve^2).
\]
In addition, recalling (\ref{e:defg}), the corresponding map $g_\ve$ satisfies
\[
g_\ve(x,v)= g(x,v) + \ve \left( \Delta y(x) +Lf''\left(\frac{v}{y'(x)}\right)\frac{v}{y'(x)^2}\Delta y'(x)\,,~ -\frac{v}{y'(x)^2}\Delta y'(x) \right) + \O(\ve^2).
\]
\v
{\bf 4.} 
Since $\bar x \in \,]a,b[\,$, we have $\bar u(\bar x)>0$. Therefore, choosing $\Tilde \psi \in C^\infty_c(y(\bar x)-\delta_\psi, y(\bar x)+\delta_\psi)$ such that 
\[
\Tilde \psi\bigl(y(\bar x)\bigr)\,=-Lf''\left(\frac{\bar u (\bar x)}{y'(\bar x)}\right) \psi'\bigl(y(\bar x)\bigr), \qquad \Tilde \psi'\bigl(y(\bar x)\bigr)\,=\, \frac{y'(\bar x)}{\bar u(\bar x))}\psi'\bigl(y(\bar x)\bigr) - Lf''\left(\frac{\bar u (\bar x)}{y'(\bar x)}\right) \psi''\bigl(y(\bar x)\bigr),
\]
we obtain 
\[
\left\{\bega{cl}
 \ds\Delta y(\bar x) +Lf''\left(\frac{\bar u(x)}{y'(\bar x)}\right)\frac{\bar u(\bar x)}{y'(\bar x)^2}\,\Delta y'(\bar x) 
 &=~ 0, \\[4mm]
 \ds-\frac{\bar u(\bar x)}{y'(\bar x)^2}\Delta y'(\bar x)&=~1.\enda\right.
\]
In turn, this implies
\bel{e:ge-g}
g_\ve\bigl(\bar x, \bar u(\bar x)\bigr)~=~ g\bigl(\bar x,\bar u(\bar x)\bigr) 
+ (0,\ve) + \O(\ve^2).
\eeq
By construction 
we now have  $g\bigl(\bar x, \bar u(\bar x)\bigr) = \bigl(z(\bar x), \bar u(z(\bar x))\bigr)$.
Moreover, by the assumption (\ref{assu}), $\bar u(\bar x)=\bar u_\ve(\bar x)$.   
This implies $g_\ve\bigl(\bar x, \bar u(\bar x))\bigr)=\bigl(z_\ve(\bar x), \bar u_\ve(z_\ve(\bar x))\bigr)$.
Thanks to \eqref{e:ge-g}, we have 
\bel{zep}z_\ve(\bar x)- z(\bar x)\,=\, \O(\ve^2),\qquad\qquad \bar u_\ve \bigl(z_\ve(\bar x)\bigr) - \bar u\bigl(z(\bar x)\bigr) = \ve + \O(\ve^2).\eeq
Since $\bar u (z(\bar x))= \bar u_\ve(z(\bar x))$ we deduce that $\bar u_\ve (z_\ve(\bar x)) - \bar u_\ve(z(\bar x)) = \ve + \O(\ve^2)$.
This contradicts the first bound in (\ref{zep}), because by (\ref{e:Lipu}) the functions $\bar u_\ve$ are uniformly Lipschitz continuous as $\ve\to 0$. This contradiction achieves the proof.
\endproof 

\section{An evolution equation for the departure pattern}
\label{sec:9}
\setcounter{equation}{0}
%

Given two functions $\vp,\psi$,
the total cost to a driver starting at time $x$ is defined as
\bel{Pudef}
\Phi^u(x)~\doteq~\vp(x)+\psi(y^u(x))\,.\eeq
Notice that here the arrival time 
$y^u(x)$ depends not only on the departure time 
$x$ but also on the entire traffic pattern $u(\cdot)$.   In particular, by (\ref{atime}) it depends on the 
behavior of drivers that have departed at times  earlier than $x$.
%

Following one of the models proposed in \cite{BLSY}
we assume that   drivers initially departing  at time $x$ switch to other times 
$\xi\in\R$, with probability proportional
to the difference in cost.  Calling $u(\theta,x)= Q_x(\theta,x)$ the rate of departures at time $x$ on day $\theta$,
in continuum time this leads to 
\bel{mod2}
{d\over d\theta} u(\theta ,x) ~=~\Lambda(u,x)~\doteq~\int u(\xi)\Big[ \Phi^u(\xi) - \Phi^u(x)\Big]_+ d\xi
-\int u(x)\Big[ \Phi^u(x)-\Phi^u(\xi)\Big]_+ d\xi\,.\eeq
Our first goal is to prove that the evolution equation (\ref{mod2})  is well posed.
Namely, 
for any  initial departure pattern 
\bel{u0} u(0,x)~=~\bar u(x)~\geq~0,\eeq
a unique solution exists, defined
for all $t\geq 0$.
Our main result is
\begin{theorem}\label{t:91}
Let the assumptions {\bf (A1)--(A2)} hold. Assume that the initial departure rate
$\bar u\in \L^1(\R)\cap\L^\infty(\R)$  has  bounded support.   Then the evolution problem (\ref{mod2}) 
has a unique solution $u=u(\theta, x)$, defined for all $\theta\geq 0$.     Its support remains uniformly bounded in time.
\end{theorem}

{\bf Proof.} {\bf 1.} We write (\ref{mod2}) as an abstract  evolution equation in the Banach space $\L^1(\R)$,
\bel{m2}
{d\over d\theta} u(\theta)~=~\Lambda\bigl(u(\theta)\bigr),\qquad\qquad u(0)=\bar u.\eeq
By the standard theory of ODEs in Banach spaces,  this Cauchy problem has unique local solution provided
that the map $u\mapsto \Lambda(u)$ is locally Lipschitz continuous w.r.t.~the $\L^1$ norm.
%
\v
{\bf 2.}  We observe that an a priori bound on the support of $u(\theta,\cdot)$ yields a bound on 
$\bigl\|u(t,\cdot)\bigr\|_{\L^\infty}$. Indeed,
assuming that the solution $u(\theta,\cdot)$ is supported inside an interval of length $R$,
by the assumptions {\bf (A2)} 
we have 
\bel{linf}\bega{rl} \ds
{d\over d\theta} \bigl\|u(\theta,\cdot)\bigr\|_{\L^\infty} &\leq~\ds \sup_x 
\int \Big[ \Phi^u(\xi) - \Phi^u(x)\Big]_+ u(\theta, \xi) \,d\xi\\[4mm]
&\ds \leq~\sup_{|x-\xi|< R} \bigl| \Phi^u(\xi) - \Phi^u(x)\bigr|\cdot    \bigl\|u(\theta,\cdot)\bigr\|_{\L^1}\\[4mm]
&\ds\leq~
C_0(1+R)   \bigl\|u(\theta,\cdot)\bigr\|_{\L^1}\,,\enda\eeq
for some constant $C_0$ depending on the Lipschitz constants $L_\vp,L_\psi$ in (\ref{costp1}) and on the maximum travel time
$\ell_{max}$  in (\ref{ellbd}).  
\v
{\bf 3.} For any $u,v\geq 0$ with $\|u\|_{\L^1}\leq K$, $\|v\|_{\L^1}\leq K$, 
by Proposition ~\ref{p:LagrHolder} we have
\bel{dLuv}\bega{l}\ds
\bigl\| \Lambda(u) - \Lambda(v)\bigr\|_{\L^1}\ds~=~\int_a^b\big| \Lambda(u,x)-\Lambda(v,x)\bigr|\, dx\\[4mm]
\ds\quad \leq ~\dint  \bigl| u(\xi)- v(\xi)\bigr| \Big[ \Phi^u(\xi) - \Phi^u(x)\Big]_+ d\xi \,dx \\[4mm]
\qquad \ds +\dint v(\xi)\bigg| \bigl[\Phi^u(x)-\Phi^u(\xi)\bigr]_+ - \bigl[\Phi^v(x)-\Phi^v(\xi)\bigr]_+ \bigg| d\xi\, dx\\[4mm]
\ds \qquad +\dint \bigl|u(x)-v(x)\bigr|\, \Big[ \Phi^u(x)-\Phi^u(\xi)\Big]_+ d\xi\, dx \\[4mm]
\ds\qquad + \dint v(x) \, \bigg| \Big[ \Phi^u(x)-\Phi^u(\xi)\Big]_+- \Big[ \Phi^v(x)-\Phi^v(\xi)\Big]_+\bigg|\, d\xi\, dx
\\[4mm]
\quad \leq~(b-a) \cdot \| u-v\|_{\L^1}\cdot \|\Phi^u\|_{\L^\infty}  + (b-a) \|v\|_{\L^\infty} \cdot 2 \|\Phi^u-\Phi^v\|_{\L^1}
\\[2mm]
\qquad +(b-a) \cdot \| u-v\|_{\L^1}\cdot \|\Phi^u\|_{\L^\infty}  + (b-a) \|v\|_{\L^\infty} \cdot 2 \|\Phi^u-\Phi^v\|_{\L^1}\,.
\enda
 \eeq
 We used here the elementary inequality 
$$ \bigg| \bigl[\Phi^u(x)-\Phi^u(\xi)\bigr]_+ - \bigl[\Phi^v(x)-\Phi^v(\xi)\bigr]_+ \bigg|~\leq~
\bigl| \Phi^u(x) - \Phi^v(x)\bigr| + \bigl| \Phi^u(\xi) - \Phi^v(\xi)\bigr| .$$

\begin{figure}[htbp]
   \centering
\includegraphics[width=1.0\textwidth]{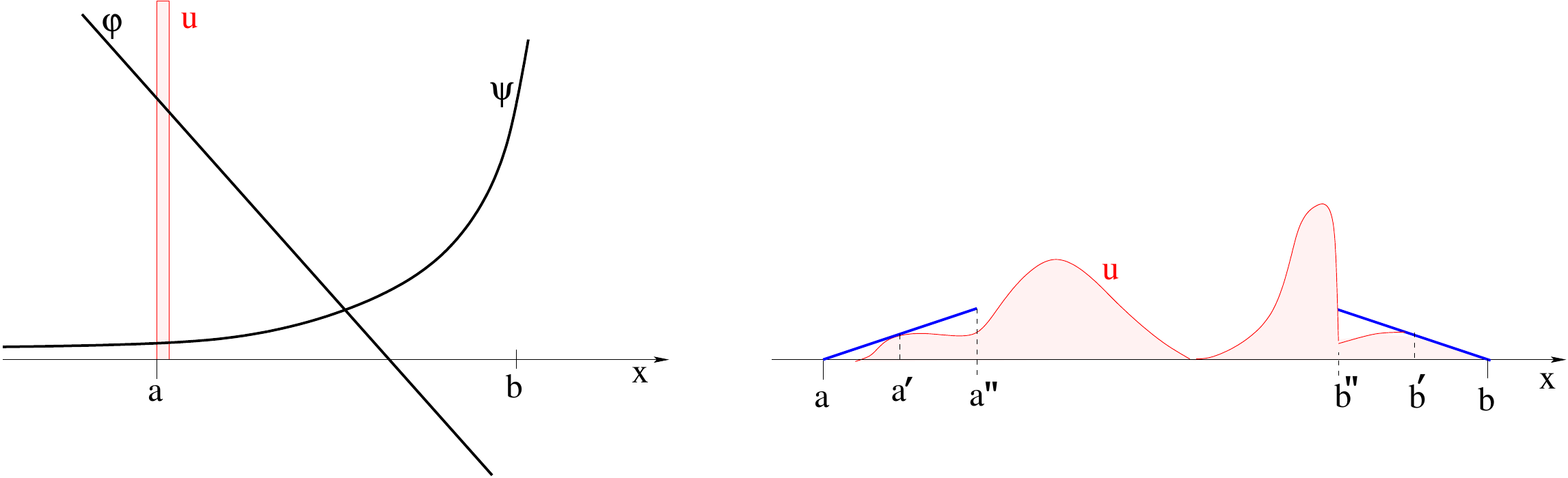}
   \caption{\small Left: an initial data supported on $[a,b]$, but such that for $\theta>0$ the support contains points $x<a$.    Right: a function $u$ in the positively invariant set $\F$ at (\ref{posinv}), with support contained in $[a,b]$.}
   \label{f:tf259}
\end{figure}
\v
{\bf 4.} In the next few steps we provide an a priori bound on the support of $u(\theta,\cdot)$.
The fact that the support remains uniformly bounded in time is not trivial.   Indeed, 
no matter how large we take the interval $[a,b]$, the set of all departure distributions supported on $[a,b]$
is NOT positively invariant for the evolution (\ref{mod2}). 

{\bf Example.} 
As shown in Fig.\ref{f:tf259}, left,  let
$$u(x)~=~\left\{\bega{cl} \kappa/\ve\quad &\hbox{if} ~~x\in [a, a+\ve],\\[2mm]
0\qquad &\hbox{otherwise.}\enda\right.$$
Then all drivers departing at time $x\in [a+\ve/2, a+\ve]$ 
arrive after they have cleared the queue at the entrance of the highway, i.e.~at time
$$y(x) ~>~a+ (x-a) {\kappa/\ve\over M}+\ell_0~\geq~ a + {\kappa\over 2M}+\ell_0\,.$$
For $\ve>0$ sufficiently small, since $\psi'(a)>0$ we obtain
$$ \vp(x) + \psi(y(x))~> ~\vp(x) + \psi\Big( a + {\kappa\over 2M}+\ell_0\Big)~>~\vp(a) + \psi(a+\ell_0)$$
for all drivers departing at time $x\in [a+\ve/2, a+\ve]$.   As a consequence, all these drivers will have an advantage in departing before time $a$.

To establish an a priori bound on the support, we
denote by $\ell_0, \ell_{max}$ a lower and an upper bound on the duration of  
a trip, as in (\ref{ell0}), (\ref{ellmax}).  These constants depend only on the flux function $f$ and on the total number $\kappa$ of drivers.
For any departure rate $ \bar u$ with $\|\bar u\|_{\L^1}=\kappa$, the corresponding cost $\Phi^{\bar u}(x)$ to a driver 
departing at time $x$ satisfies the bounds
\bel{Phib2} \Phi^0(x)~\doteq ~\vp(x) + \psi(x+\ell_0) ~\leq~\Phi^{\bar u}(x)~\leq ~\vp(x) + \psi(x+\ell_{max}) 
~\doteq~\Phi^{max}(x).\eeq
By the assumptions {\bf (A2)} on $\vp,\psi$, the map $x\mapsto \Phi^0(x)$ is convex and satisfies
$$\lim_{x\to \pm \infty} \Phi^0(x)\,=\,+\infty.$$
In particular, for any  sufficiently large time $b$ there exists a unique 
time $a<b$ such that 
$$\Phi^0(a)~=~\Phi^0(b).$$

Consider the upper bound on the second derivative
$$\omega_2~\doteq~\sup_{u\in [0,M/2]} f''(u),$$
and denote by  $L_{f/u}$ the Lipschitz constant of the function $f(u)/u$, restricted to the interval $u\in [0, M/4]$.
Choose $\delta>0$ so small that
\bel{delp}\kappa\,\leq \,{3M\over 16 \delta} \,,\qquad 
\delta\,<\,{1\over 2\omega_2 L}\,,\qquad\qquad   {1\over 2\delta} \,>\,\ell_{max}\,, \eeq
and, in addition,
\bel{delp2} \delta < \frac{3M}{16L(f(M)-f(M/4))},  \qquad \frac{L_\Phi}{8\delta} \geq L_\psi (\ell_{max}-\ell_0), \eeq
\bel{delp3}  \delta < \frac{2}{L\omega_2M},  \qquad \delta < \frac{L_\Phi}{L_\psi L L_{f/u}M}.\eeq
We introduce the additional points
$$ a<a'<a''<b''<b'<b,$$
determined by the identities
\bel{idab} b''= b-\delta^{-1},\qquad \Phi^0(a'') = \Phi^0(b''), \qquad a' = {a+a''
\over 2},\qquad b'={b+b''\over 2}\,.\eeq
We claim that, if the interval $[a,b]$ is chosen large enough, then  the family of all departure rates
\bel{posinv}\bega{rl} \F&\ds\doteq~\bigg\{ u:[a,b]\mapsto\R_+\,;~~\int_a^b u(x)\, dx=\kappa,
\\[2mm]
&\ds\qquad u(x)\leq {M(x-a)\over 4(a''-a)}  ~~\hbox{if} ~~x\in [a, a''],\qquad u(x)\leq  
{M(b-x)\over 4(b-b'')}~~\hbox{if~~} x\in [b'',b]\bigg\}
\enda\eeq
 (see Fig.~\ref{f:tf259}, right)
is positively invariant. 
The positive invariance of the set $\F$ will be proved by showing that, for any  departure rate $u\in\F$, one has:
\begi
\item[(i)] If $x\in \,]-\infty, a]\cup [b, +\infty[\,$, then $\Lambda^u(x)=0$.
 \item[(ii)] If $x\in [b', b[$ and $u(x)=  {M(b-x)\over 4(b-b'')}$, then $\Lambda^u(x)< 0$.
\item[(iii)] If $x\in \,]a, a']$ and $u(x)=  {M(x-a)\over 4(a''-a)}$, then $\Lambda^u(x)<0$.
 \item[(iv)] If $x\in [b'', b'[$ and $u(x)=  {M(b-x)\over 4(b-b'')}$, then $\Lambda^u(x)< 0$.
 \item[(v)] If $x\in \,]a', a'']$ and $u(x)=  {M(x-a)\over 4(a''-a)}$, then $\Lambda^u(x)< 0$.
 \endi
 \v
 {\bf 5.}
 As a preliminary, in this step we analyze the solution to (\ref{claw1}) 
with  an initial data $\bar u\in \F$. We seek an upper  bound on the cost $\Phi^{\bar u}$, 
sharper than the one in (\ref{Phib2}).

\begin{figure}[htbp]
   \centering
\includegraphics[width=0.7\textwidth]{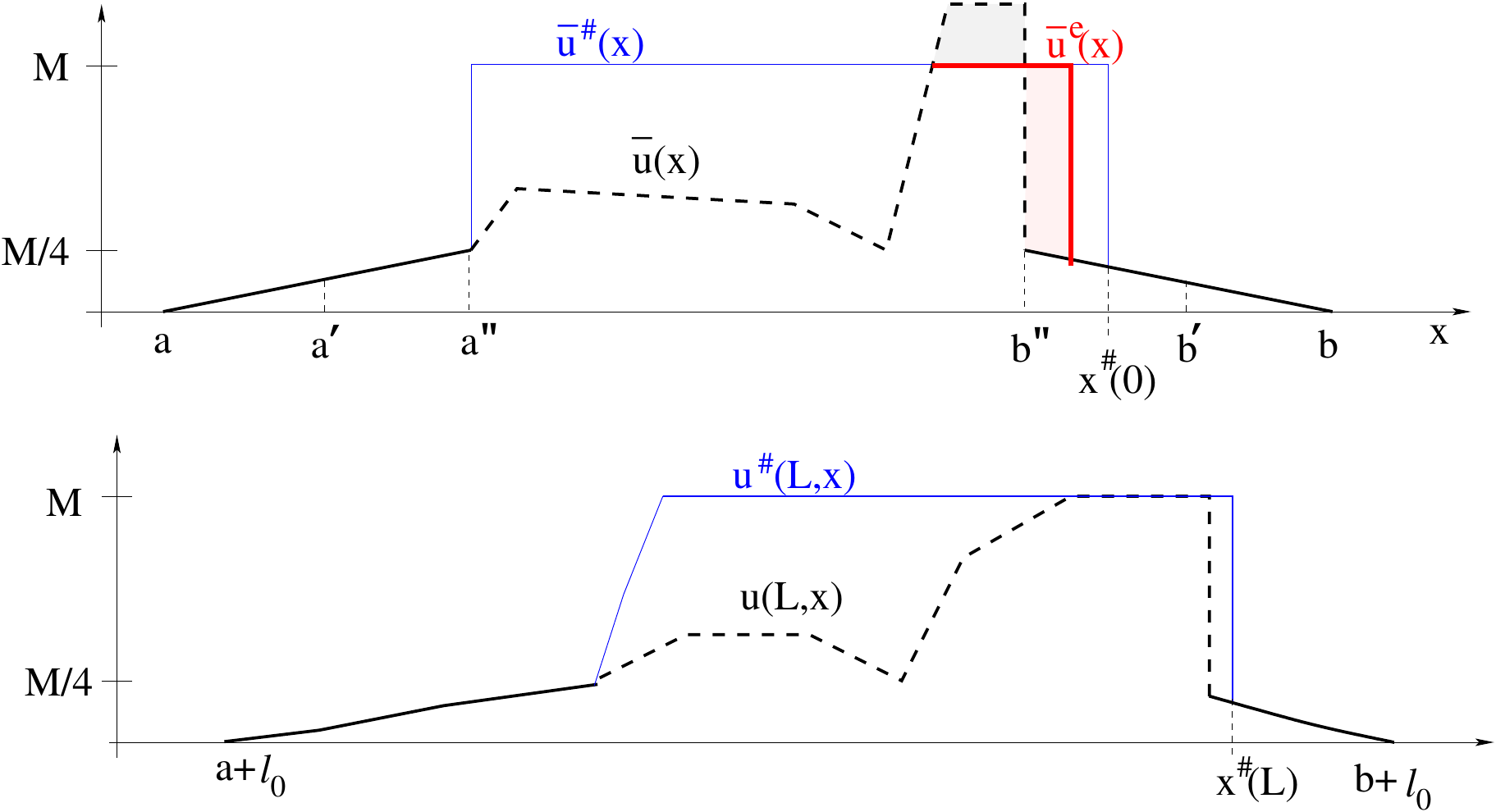}
   \caption{\small Top: a departure rate $\bar u$ as in (\ref{posinv}), and the corresponding effective 
 rate $\bar u^e$ in (\ref{eflux}).   The function $\bar u^\sharp\geq \bar u^e$ provides an upper bound for which
 the solution can be explicitly computed.
   Bottom: the rates of arrivals $u(L,x)\leq u^\sharp(L,x)$ at the end of the highway.}
   \label{f:tf268}
\end{figure}

Consider a departure rate $\bar u\in \F$.   If $\bar u$ takes values larger than the maximum flux $M$, 
we need to consider the effective rate $\bar u^e$, as shown in Fig.~\ref{f:tf268}, top. 
Since the shaded area is $\leq \kappa$, if the first inequality in (\ref{delp}) holds,
this effective rate will satisfy
\bel{erate}
\bar u^e(x)~=~\bar u(x)\qquad \hbox{for} \quad x\in \,]-\infty, a'']~\cup~\left[b-{3\over 4\delta}, ~+\infty\right[\,.\eeq

To estimate the arrival time of various drivers, we use a comparison argument in terms of the departure rate
\bel{bush}\ds\bar u^\sharp(x)~=~\left\{ \bega{cl}\ds{M (x-a)\over 4(a''-a)}  \qquad &\hbox{if} ~~x\in [a, a''],\\[4mm]
\ds {M(b-x)\over 4(b-b'')}\qquad &\hbox{if~~} x\in [x^\sharp(0),b],\\[4mm]
 M \qquad  &\hbox{if~~} x\in \,\bigl]a'',\, x^\sharp(0)\bigr[\,,\\[4mm]
0 \qquad  &\hbox{if~~} x\notin[a,b].\enda\right.\eeq
Here the point $x^\sharp(0)$ is determined by the identity
$$\int_{b''}^{x^\sharp(0)}  \left(M-{M (b-x)\over 4(b-b'')}\right) dx~=~\kappa.$$
The above construction implies $\bar u^\sharp(x)\geq \bar u^e(x)$.

The solution of (\ref{claw1}) with initial data $\bar u^\sharp$  can be constructed by characteristics:
$$u^\sharp\Big(L, \xi+ L f'\bigl(\bar u^\sharp(\xi)\bigr)\Big)\,=\,\bar u^\sharp(\xi)\qquad\qquad \hbox{for $\xi< a''$ and for $\xi>b'$.}$$
We observe that this solution also contains a shock, say located at $x=x^\sharp(t)$.
Since the right state across the shock   is $\leq M/4$, this shock has speed
$$\dot x^\sharp(t)~\leq~{f(M)-f(M/4)\over 3M/4}\,.$$
Choosing $\delta>0$ as in \eqref{delp} we achieve
$$x^\sharp(L) ~<~ b'~=~b-{1\over 2\delta}\,.$$
Moreover (see Fig.~\ref{f:tf268}, bottom)
$$u^\sharp(L, x)~=~0\qquad\qquad \hbox{if}~~x\notin [a+\ell_0, \, b+\ell_0].$$

By the method of characteristics we obtain the inequalities
$$u^\sharp(L,x)~\geq~\left\{ \bega{cl}\lambda (x-a-\ell_0)  \qquad  &x\in [a+\ell_0, a''+\ell_0],\\[1mm]
\mu (b+\ell_0 -x)  \qquad &x\in [b'+\ell_0, b+\ell_0],\enda\right.$$
for suitable constants 
$$0~<~\lambda~\leq~\delta~\leq~\mu.$$

If the departure rate were $\bar u^\sharp$, 
the driver starting at $x\in [a,a']$ would arrive at time $y^\sharp(x)$ implicitly determined by
$$\int_a^x {M(s-a)\over 4(a''-a)}\, ds~=~\int_{a+\ell_0}^{y(x)} u^\sharp(L, s) \, ds~\geq~\int_{a+\ell_0}^{y(x)} \lambda (s-a-\ell_0) \, ds.$$
This yields
$${\lambda\over 2}\bigl(y(x)-a-\ell_0\bigr)^2 ~\leq~ {M\over 8(a''-a)} (x-a)^2,$$
and hence 
\bel{depta} y(x)-a-\ell_0~\leq~\sqrt{M\over 4\lambda(a''-a)} \cdot(x-a)~\doteq~\Hat \lambda (x-a),\eeq
\bel{dPa} \Phi^0(x)~\leq~\Phi^{\bar u}(x)~\leq~\Phi^{\bar u^\sharp}(x)~\leq~\Phi^0(x) + L_\psi \,\Hat\lambda (x-a),
\qquad\quad\forall  x\in [a, a'].\eeq
An entirely similar argument yields
\bel{dPb}  \Phi^0(x)~\leq~ \Phi^{\bar u}(x)~\leq~\Phi^{\bar u^\sharp}(x)~\leq~\Phi^0(x) + L_\psi \,\Hat\mu (b-x),
\qquad\quad\forall  x\in [b', b].\eeq

Using (\ref{dPa})-(\ref{dPb}) we shall prove the claims (i)--(v),  assuming that $b$ is sufficiently large.
\begin{figure}[htbp]
   \centering
\includegraphics[width=0.7\textwidth]{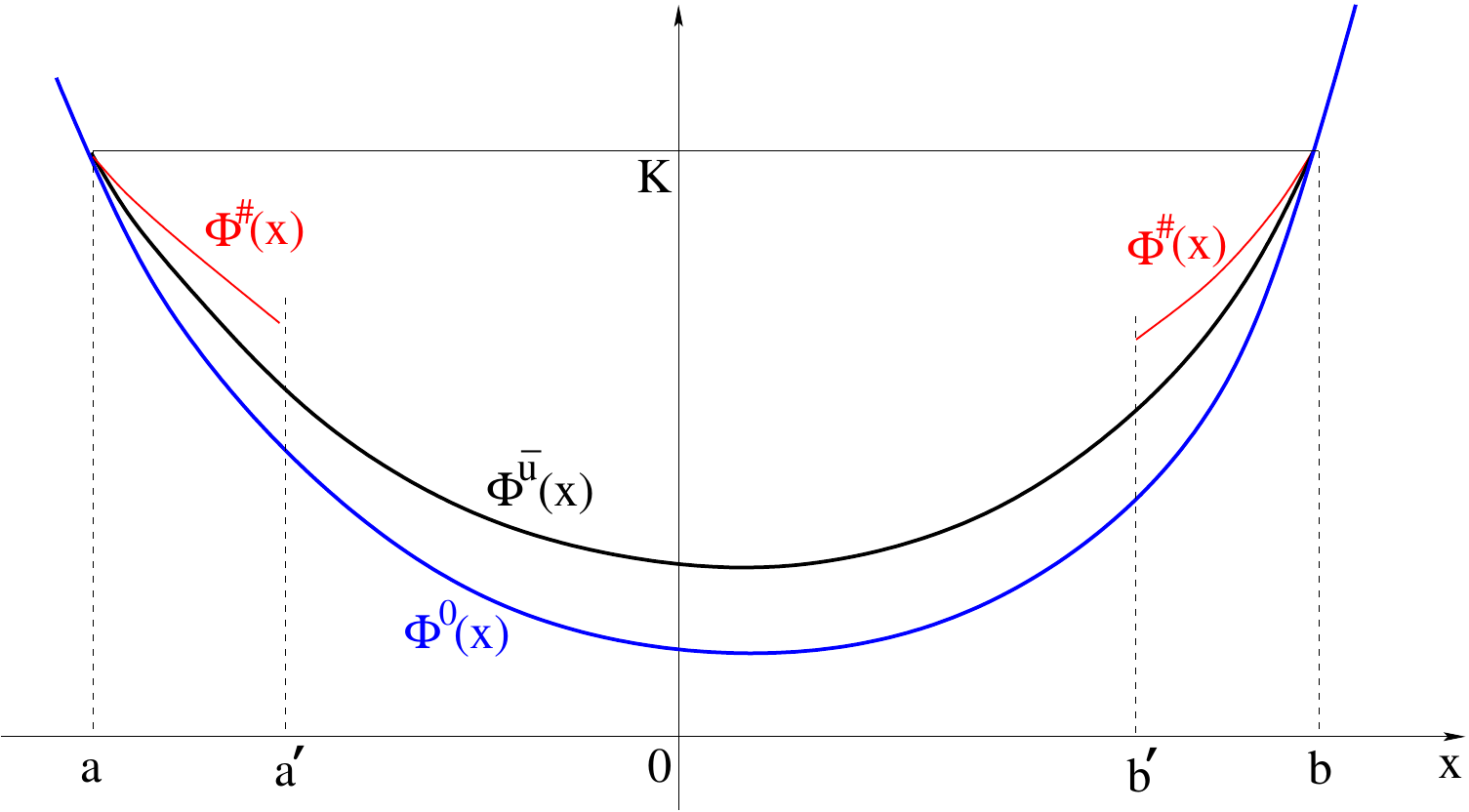}
   \caption{\small The total cost $\Phi^{\bar u}(x)$ to drivers with departure time $x$. 
   Here $\Phi^0$ is a lower bound, corresponding to an empty road. The upper bound $\Phi^\sharp \doteq \Phi^{\bar u^\sharp}$ is obtained by a comparison with the departure rate (\ref{bush}).}
   \label{f:tf267}
\end{figure}
\v
{\bf 6.} To prove (i),  let $\bar u\in \F$ and call
\bel{Mab}K~\doteq~\Phi^0(a)~=~\Phi^0(b).\eeq
Since 
$$\Phi^{\bar u}(x)~\geq~\Phi^0(x)~\geq~K\qquad\qquad\forall x\notin \,]a,b[\,,$$
it suffices to show that 
\bel{Psma}\Phi^{\bar u}(x)~\leq~K\qquad\qquad\forall x\in [a,b].\eeq

We first prove \eqref{Psma} for $x \in [a',b']$.
By \eqref{LP1} and \eqref{LP2}, it is possible to choose $K$ sufficiently large so that 
$$\Phi^0_x (x) \le -{L_\varphi \over 2} \quad \forall x \le a'', \qquad \Phi^0_x (x) \ge {L_\Phi \over 2}\quad \mbox{and} \quad \psi'(x) > \frac{L_\psi}2 \quad \forall x \ge b''.
$$
Therefore
\bel{a''-a}
\frac1\delta = b - b'' \le \frac2{L_\Phi}(\Phi^0(b)-\Phi^0(b'')) = \frac2{L_\Phi}(\Phi^0(a)-\Phi^0(a'')) \le \frac{2L_\varphi}{L_\Phi} (a''-a)
\eeq
In particular for $x \in [a',b']$ it holds 
$$\Phi^{\bar u}(x)~\leq~ \Phi_0(x) +L_\psi(\ell_{max}-\ell_0)~\leq~\max\{ \Phi^{\bar u}(a'),\Phi^{\bar u}(b')\}
+ L_\psi(\ell_{max}- \ell_0).$$ 
We have
$$
\Phi^{\bar u}(a') \le \Phi^0(a') + L_\psi (\ell_{max}-\ell_0) \le \Phi^0(a) - \frac{L_{\varphi}}2(a'-a) + L_\psi(\ell_{max}-\ell_0) \le K - \frac{L_\Phi}{8 \delta} + L_\psi (\ell_{max}-\ell_0) \le K
$$
by \eqref{delp2}, and similarly
$$
\Phi^{\bar u}(b') \le \Phi^0(b') + L_\psi (\ell_{max}-\ell_0) \le K - \frac{L_\Phi}{2}(b-b') +  L_\psi (\ell_{max}-\ell_0) = K - \frac{L_\Phi}{4\delta} + L_\psi  (\ell_{max}-\ell_0) \le K.
$$
We now prove \eqref{Psma} for $x \in [a,a']$. Since $\Phi^{\bar u}(x) \le \Phi^{\bar u^\sharp} (x)$, it is sufficient to prove that $\Phi^{\bar u^\sharp} (x) \le K$.
It holds
$$
\Phi^{\bar u^\sharp}(x) = \varphi (x) + \psi (y^\sharp(x))\le \Phi^0(x) + L_\psi (y^\sharp(x)-x-\ell_0) \le K - \frac{L_{\varphi}}2(x-a) + L_\psi (y^\sharp(x)-x-\ell_0)
$$
We now prove that $ y^\sharp(x)-x-\ell_0 \le \frac{L_\varphi}{2L_\psi}(x-a)$. Notice that $y^\sharp (x)= y^\sharp(L,x)$ where
$$
y^\sharp(t,x)= x + \int_0^t \frac{f(u(s,y^\sharp(s,x)))}{u(s,y^\sharp(s,x))}ds \qquad \mbox{for }t \in [0,L].
$$
By the method of characteristics, for every $s \in [0,L]$ it holds $u(s,y^\sharp(s,x)) \in [0, \bar u(x)]$, therefore 
$$
\begin{array}{rl}
y^\sharp (L,x) = &\ds ~ x + \ell_0 + \int_{0}^L \left(\frac{f(u(s,y^\sharp(s,x)))}{u(s,y^\sharp(s,x))} - f'(0)\right) ds \le  x + \ell_0 + L  L_{f/u} \bar u (x) \\[4mm]
\le &\ds ~ x+\ell_0 + L  L_{f/u} \frac{M}{4(a''-a)}(x-a) \le ~ x+\ell_0 +  \frac{L_\varphi}{2L_\psi}(x-a)
\end{array}
$$
provided $\delta < \frac{L_\Phi}{L_\psi L  L_{f/u} M}$,   thanks to \eqref{a''-a}.

We prove \eqref{Psma} for $x \in [b',b]$.
$$
\Phi^{\bar u^\sharp}(x) = \varphi (x) + \psi (y^\sharp(x))\le \Phi^0(x) + L_\psi (y^\sharp(x)-x-\ell_0) \le K - \frac{L_{\Phi}}2(b-x) + L_\psi (y^\sharp(x)-x-\ell_0)
$$
We now prove that $ y^\sharp(x)-x-\ell_0 \le \frac{L_\Phi}{2L_\psi}(b-x)$. 
By the method of characteristics, for every $s \in [0,L]$ it holds $u(s,y^\sharp(s,x)) \in [\bar u(x), \bar u (\tilde x)]$ for $\tilde x \in [b'',b]$ such that $\bar u^{\sharp}(L,y^\sharp(x))=\bar u^\sharp(\tilde x)$.
In particular $y^\sharp (x) \ge x + \frac{f(\bar u(x))}{\bar u(x)}L$ and therefore by the method of characteristics $\bar u^{\sharp}(L,y^\sharp(x))=\bar u^\sharp(\tilde x)$ implies that 
$$L\left(f'(\bar u^\sharp (\tilde x))- \frac{f(\bar u^\sharp(x))}{\bar u^\sharp(x)}\right) ~\ge ~x-\tilde x ~= ~\frac{4}{M\delta}(\bar u^{\sharp}(\tilde x) - \bar u^\sharp (x)) .$$
Therefore
$$
\frac4{M\delta}(\bar u^{\sharp}(\tilde x) - \bar u^\sharp (x)) ~ \le ~  L\left(f'(\bar u^\sharp (\tilde x))- \frac{f(\bar u^\sharp(x))}{\bar u^\sharp(x)}\right)~ \le ~L \omega_2 \bar u^\sharp (\tilde x) ,
$$
which implies
$$
 \bar u^\sharp(\tilde x) \le \frac1{1-\frac{L\omega_2M\delta}4}\bar u^\sharp(x) \le 2 \bar u^\sharp(x) \qquad \mbox{for} \quad \delta < \frac{2}{L\omega_2M}.
$$

We can then estimate
$$
\begin{array}{rl}
y^\sharp(x)-x+\ell_0 &= \ds ~ x + \ell_0 + \int_0^L \left(\frac{f\bigl(u(s,y^\sharp(s,x))\bigr)}{u(s,y^\sharp(s,x))}-f'(0)\right)ds
\\[4mm]
&\ds\le~ x + \ell_0 + 2L L_{f/u}\bar u^\sharp(x) \\[4mm]
&= ~ \ds x + \ell_0 +  \frac12L L_{f/u} M\delta (b-x) ~\le~  \frac{L_\Phi}{2L_\psi}(b-x),
\end{array}
$$ 
provided that $\delta < \frac{L_\Phi}{L_\psi L  L_{f/u}  M}$ as in  (\ref{delp3}).

%
%
\v
{\bf 7.} 
To prove (iv), 
we observe that  for $x\in [b'',b']$  by (\ref{LP2}) one has
\bel{Pblb}\Phi^0(x)~\geq~\Phi^0(b'') ~=~ \Phi^0(b-\delta^{-1})~\geq~\Phi^0(b) - L_{\Phi}\delta^{-1}.\eeq
 For $x\in [b'', b']$, assuming $\bar u(x)= {M(b-x)\over 4(b-b'')}$,
 the positive term on the right hand side of  (\ref{m2}) can thus be bounded as
\bel{gain}\Lambda^{\bar u}_+(x)~\leq~\int_a^b \bar u(\xi)\Big[ \Phi^{\bar u}(\xi)-\Phi^{\bar u}(x)\Big]_+\, d\xi~\leq~\kappa 
\bigl[\Phi^0(b) - \Phi^0(b'')\Big]~\leq~\kappa\cdot L_\Phi \,\delta^{-1}.
\eeq
On the other hand, the negative term can be estimated as
\bel{loss}\Lambda^{\bar u}_-(x)~\doteq~\int \bar u(x)\Big[ \Phi^u(x)-\Phi^{ \bar u}(\xi)\Big]_+\, d\xi~\geq~{M\over 8} \cdot\int\Big[ \Phi^0(b-\delta^{-1}) -
\Phi^{max}(\xi)\Big]_+\, d\xi.
\eeq
Observe that the right hand side of (\ref{gain}) is uniformly bounded regardless of the choice of $K$ and hence of 
the interval $[a,b]$ (see Fig.~\ref{f:tf272}). On the other hand, the loss term (\ref{loss}) becomes arbitrarily large as $K\to \infty$.
By choosing $K$ large enough, the difference $\Lambda^{\bar u}_+(x)- \Lambda^{\bar u}_-(x)$ can thus be rendered arbitrarily large and negative, uniformly for all $x\in [b'', b']$. This proves (iv).   The proof of (v) is entirely similar.
\begin{figure}[htbp]
   \centering
\includegraphics[width=0.6\textwidth]{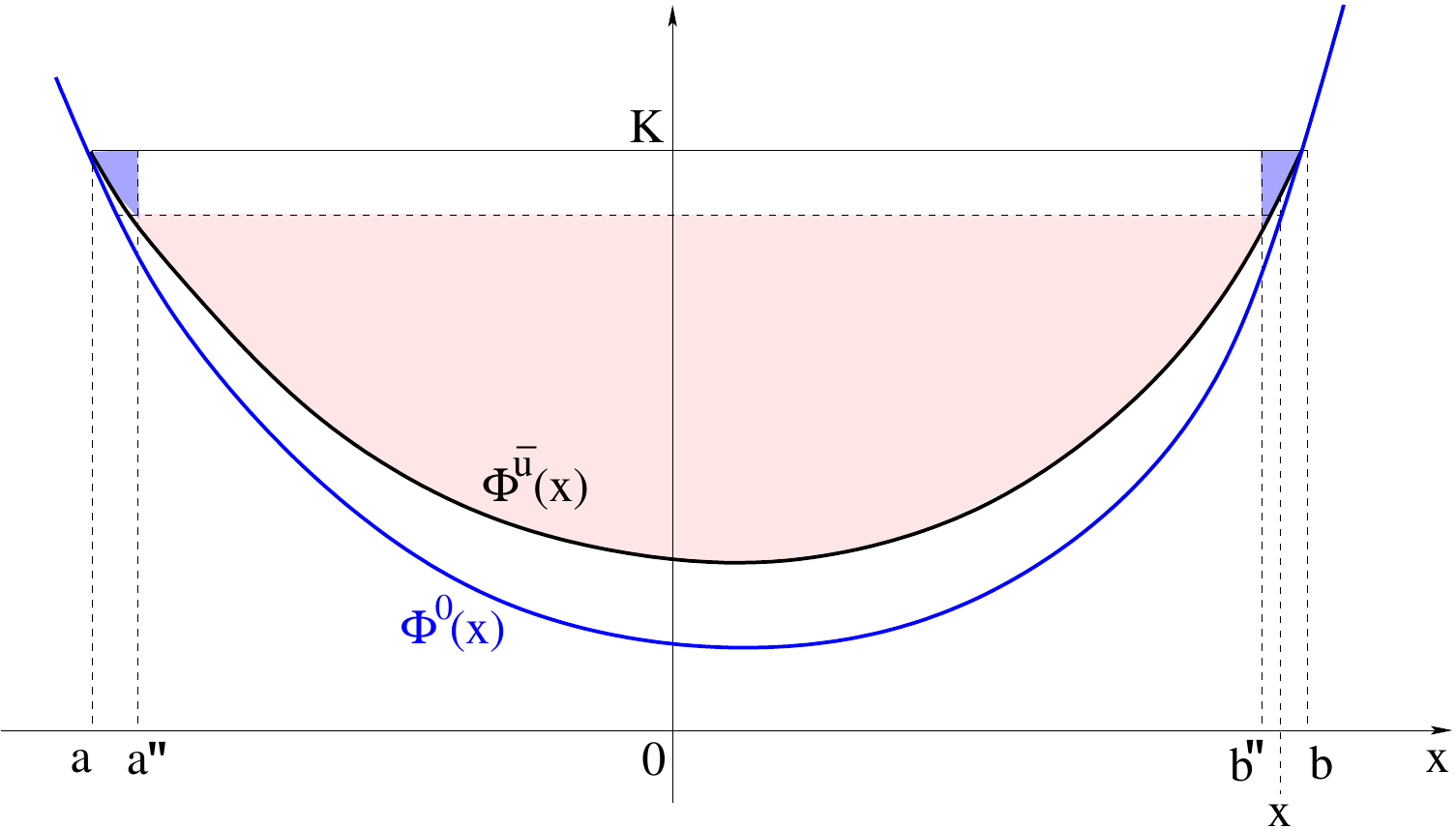}
   \caption{\small The positive and negative terms in (\ref{gain})-(\ref{loss}) are proportional to the 
  two shaded areas.  As $K\to+\infty$, the gain term remains uniformly bounded while the loss term becomes arbitrarily large.} \label{f:tf272}
\end{figure}
\v
{\bf 8.}  To prove (ii), let $x\in [b',b[\,$ and assume $\bar u(x)= {M(b-x)\over 4(b-b'')}$.
By  \eqref{LP2}, in this case the gain term can be bounded by
\bel{gain'} \Lambda^{\bar u}_+(x)~=~\int_a^b \bar u(\xi) \Big[ \Phi^{\bar u}(\xi) - \Phi^{\bar u} (x)\Big]_+\, d\xi
~\leq~\kappa\Big[ \Phi^0(b)-\Phi^0(x)\Big]_+~\leq~\kappa \cdot L_\Phi (b-x).
\eeq
On the other hand, the loss term is
\bel{loss'} \Lambda^{\bar u}_- (x)~=~\bar u(x) \int_a^b \Big[ \Phi^{\bar u}(x) - \Phi^{\bar u} (\xi)\Big]_+\, d\xi~\geq~
{M\over 4(b-b'')} (b-x) \cdot \int_a^b\Big[ \Phi^0(x) - \Phi^{max}(\xi)\Big]_+\, d\xi.
\eeq
By choosing $K$ large enough
the integral on the right hand side of (\ref{loss'}) can be rendered arbitrarily large, uniformly for $x\in [b',b]$.
This achieves the desired inequality  $\Lambda^{\bar u}_- (x)> \Lambda^{\bar u}_+ (x)$,
for all $x\in [b', b[\,$. The proof of (iii) is entirely similar.
\v
 {\bf 9.}  Let an initial departure rate $u(0,\cdot)=\bar u\in \L^1(\R)\cap\L^\infty(\R)$
be given,  with compact support.
 For any time interval $\theta\in [0,T]$, by the uniform boundedness of the support of $u(\theta,\cdot)$ 
 proved in steps {\bf 4--8}, we can use 
 (\ref{linf}) and obtain
 \bel{lifb}
 \bigl\| u(\theta, \cdot)\bigr\|_{\L^\infty}~\leq~ \exp \bigl\{ C_0(1+R) T\bigr\}\cdot l\| \bar u\|_{\L^\infty}
 \qquad\qquad\forall \theta\in [0,T].
 \eeq
If now $u,v\in \L^1(\R)\cap\L^\infty(\R)$ satisfy the same bounds as above, then 
 \bel{PL1}\bigl\|\Phi^u-\Phi^v\bigr\|_{\L^1}~=~\O(1)\cdot \|y^u - y^v\|_{\L^1}~=~\O(1)\cdot \|u-v\|_{\L^1}\,.\eeq
 Indeed, the first inequality follows  from (\ref{Pudef}) and the Lipschitz continuity of the cost function
 $\psi$.  The second inequality follows from (\ref{Liparr}).
 
In turn,  (\ref{dLuv}) shows the Lipschitz continuity of the map $u\mapsto \Lambda(u)$, restricted to
a domain of functions  uniformly bounded in $\L^1(\R)\cap\L^\infty(\R)$.
By the standard theory of ODEs in Banach spaces (see for example \cite{JD, RM}), this yields the global existence of a unique solution
to the Cauchy problem (\ref{m2}).
\endproof
\v

{\bf Acknowledgments.}  The research of A.\,Bressan and D.\,L.\,Harmon  was partially supported by NSF with
grant  DMS-2306926, ``Regularity and approximation of solutions to conservation laws".
The research of E.\ Marconi was partially supported by Gnampa - INDAM project 2026 ``Discrete to continuum limit in traffic and related models".

The authors acknowledge the use of ChatGPT to check for misprints and inaccuracies.
The manuscript was written entirely by the authors, who take full responsibility for its mathematical content and correctness.

\v

\end{document}